\documentclass{cimart}

\newcommand{\C}{\mathbb{C}}
\newcommand{\EE}{\mathcal{E}}
\newcommand{\qbinom}{\genfrac{[}{]}{0pt}{}}
\newcommand{\rbhs}{\mathcal{R}}
\newcommand{\RR}{\mathcal{R}}

\DeclareMathOperator{\ope}{e}
\DeclareMathOperator{\opE}{E}
\DeclareMathOperator{\opF}{F}
\DeclareMathOperator{\oph}{h}
\DeclareMathOperator{\opP}{P}
\DeclareMathOperator{\opR}{R}
\DeclareMathOperator{\opr}{r}
\DeclareMathOperator{\opS}{S}
\DeclareMathOperator{\opT}{T}

\title{Deformed homogeneous polynomials and the deformed $q$-exponential operator}

\authors{Ronald Orozco López}

\authorinfo{University of Los Andes, Colombia}{rj.orozco@uniandes.edu.co}

\abstract{In this paper, we introduce the deformed homogeneous polynomials $\mathrm{R}_{n}(x,y;u|q)$. These polynomials generalize some classical polynomials: the Rogers-Szeg\"o polynomials $\mathrm{h}_{n}(x|q)$, the generalized Rogers-Szeg\"o polynomials $\mathrm{r}_{n}(x,y)$, the Stieltjes-Wigert polynomials $\mathrm{S}_{n}(x;q)$, among others. Basic properties of the polynomial $\mathrm{R}_{n}$ are given, along with recurrence relations, its $q$-difference equation, and representations. Generating functions for the polynomials $\mathrm{R}_{n}(x,y;u|q)$ are given. These functions include generalizations of the Mehler and Rogers formulas. In addition, generalizations of the $q$-binomial formula and the Heine transformation formula are obtained. These results are obtained via the $u$-deformed $q$-exponential operator $\mathrm{E}(yD_{q}|u)$, defined here. From this operator, we obtain for free the operators $\opT(yD_{q})$ the Chen, $\mathrm{R}(yD_{q})$ of Saad, $\mathcal{E}(yD_{q})$ of Exton, and $\mathcal{R}(yD_{q})$ of Rogers-Ramanujan when $u=1,q,\sqrt{q},q^2$, respectively. We introduce the deformed basic hypergeometric series ${}_{r}\Phi_{s}$, a generalization of the classical basic hypergeometric series. New transformation formulas for basic hypergeometric series are obtained.}



\acknowledgments{The author is grateful to Professor Roberto Costas for the helpful comments and suggestions, which significantly improved the paper.}

\VOLUME{34}
\YEAR{2026}
\ISSUE{2}
\NUMBER{7}
\DOI{https://doi.org/10.46298/cm.16723}
\editinfo{October 16, 2025}{January 20, 2026}{Tiago Macedo, Jaqueline Mesquita, Rafael Potrie, and Mariel Saez}

\begin{document}   

\section{Introduction}

We begin with some notation and terminology for basic hypergeometric series \cite{gasper}. The $q$-shifted factorial is defined by
\begin{align*}
    (a;q)_{n}&=\begin{cases}
        1&\text{ if }n=0;\\
    \prod_{k=0}^{n-1}(1-q^{k}a),&\text{ if }n\neq0,\\
    \end{cases}\hspace{1cm}q\in\C,\\
    (a;q)_{\infty}&=\lim_{n\rightarrow\infty}(a;q)_{n}=\prod_{k=0}^{\infty}(1-aq^{k}),\hspace{1cm} \vert q\vert<1.
\end{align*}
%For $\vert q\vert>1$ and $x\neq0$,
%\begin{equation*}
 %   (x;q^{-1})_{n}=q^{-\binom{n}{2}}(-x)^{n}(x^{-1};q)_{n}.
%\end{equation*}
The multiple $q$-shifted factorials are defined by
\begin{align*}
    (a_{1},a_{2},\ldots,a_{m};q)_{n}&=(a_{1};q)_{n}(a_{2};q)_{n}\cdots(a_{m};q)_{n},\hspace{0.3cm}q\in\C,\\
    (a_{1},a_{2},\ldots,a_{m};q)_{\infty}&=(a_{1};q)_{\infty}(a_{2};q)_{\infty}\cdots(a_{m};q)_{\infty},\hspace{0.3cm}\vert q\vert<1.
\end{align*}
Some useful identities for $q$-shifted factorial:
\begin{align}
    (a;q)_{n}&=\frac{(a;q)_{\infty}}{(aq^n;q)_{\infty}},\ aq^n\neq q^{-k},\ k=1,2,3,\ldots,\label{eqn_iden1}\\
    (a;q)_{n+k}&=(a;q)_{n}(aq^{n};q)_{k},\label{eqn_iden2}\\
    (a;q)_{2n}&=(a;q^2)_{n}(aq;q^2)_{n},\label{eqn_iden3}\\
    (a^2;q^2)_{n}&=(a;q)_{n}(-a;q)_{n}.\label{eqn_iden4}
\end{align}
In our work, we will use the identities for binomial coefficients:
\begin{align*}
    \binom{n+k}{2}&=\binom{n}{2}+\binom{k}{2}+nk,\\
    \binom{n-k}{2}&=\binom{n}{2}+\binom{k}{2}+k(1-n).
\end{align*}
The $q$-binomial coefficient is defined by
\begin{equation*}
\qbinom{n}{k}_{q}=\frac{(q;q)_{n}}{(q;q)_{k}(q;q)_{n-k}}.
\end{equation*}
The $q$-binomial coefficient satisfies
\begin{align}
    \qbinom{n+1}{k}_{q}&=\qbinom{n}{k}_{q}+q^{n+1-k}\qbinom{n}{k-1}_{q}=
    q^{k}\qbinom{n}{k}_{q}+\qbinom{n}{k-1}_{q},\label{eqn_pascal}\\
    \qbinom{n}{k}_{q}&=\frac{(q^{-n};q)_{k}}{(q;q)_{k}}(-q^n)^{k}q^{-\binom{k}{2}}.\nonumber
\end{align}
The ${}_r\phi_{s}$ basic hypergeometric series is defined by
\begin{equation*}
    {}_r\phi_{s}\left(
    \begin{array}{c}
         a_{1},a_{2},\ldots,a_{r} \\
         b_{1},b_{2},\ldots,b_{s}
    \end{array}
    ;q,z
    \right)=\sum_{n=0}^{\infty}\frac{(a_{1},a_{2},\ldots,a_{r};q)_{n}}{(q;q)_{n}(b_{1},b_{2},\ldots,b_{s};q)_{n}}\Big[(-1)^{n}q^{\binom{n}{2}}\Big]^{1+s-r}z^n.
\end{equation*}
In this paper, we will frequently use the $q$-binomial theorem:
\begin{equation}\label{eqn_qbin_the}
    {}_1\phi_{0}\left(\begin{array}{c}
         a\\
         - 
    \end{array};q,-z\right)=\frac{(az;q)_{\infty}}{(z;q)_{\infty}}=\sum_{n=0}^{\infty}\frac{(a;q)_{n}}{(q;q)_{n}}z^{n}.
\end{equation}
The $q$-exponential $\ope_{q}(z)$ is defined by
\begin{equation*}
    \ope_{q}(z)=\sum_{n=0}^{\infty}\frac{z^n}{(q;q)_{n}}={}_1\phi_{0}\left(\begin{array}{c}
         0\\
         - 
    \end{array};q,z\right)=\frac{1}{(z;q)_{\infty}}.
\end{equation*}
Another $q$-analogue of the classical exponential function is
\begin{equation*}
    \opE_{q}(z)=\sum_{n=0}^{\infty}q^{\binom{n}{2}}\frac{z^n}{(q;q)_{n}}={}_0\phi_{0}\left(\begin{array}{c}
         -\\
         - 
    \end{array};q,-z\right)=(-z;q)_{\infty}.
\end{equation*}
For all $u\in\C$, we define the following function, which we will call the deformed $q$-exponential function,
\begin{equation*}
    \ope_{q}(z,u)=
    \begin{cases}
        \sum_{n=0}^{\infty}u^{\binom{n}{2}}\frac{z^{n}}{(q;q)_{n}}&\text{ if }u\neq0;\\
        1+\frac{z}{1-q}&\text{ if }u=0.
    \end{cases}
\end{equation*}
Some deformed $q$-exponential functions are:
\begin{align*}
    \ope_{q}(z,1)&=e_{q}(z)=\frac{1}{(z;q)_{\infty}},\ \vert z\vert<1,\\
    \ope_{q}(-z,q)&=\opE_{q}(-z)=(z;q)_{\infty},\ z\in\C,\\
    e_{q}(z,\sqrt{q})&=\mathcal{E}_{q}(z)=\sum_{n=0}^{\infty}q^{\frac{1}{2}\binom{n}{2}}\frac{z^n}{(q;q)_{n}}={}_{1}\phi_{1}\left(\begin{array}{c}
         0\\
         -\sqrt{q}
    \end{array};\sqrt{q},-z\right),\ z\in\C,\\
    \ope_{q}(qz,q^2)&=\mathcal{R}_{q}(z)=\sum_{n=0}^{\infty}q^{n^2}\frac{z^n}{(q;q)_{n}},\ z\in\C,
\end{align*}
where $\mathcal{E}_{q}(z)$ is the Exton $q$-exponential function and  $\mathcal{R}_{q}(z)$ is the Rogers-Ramanujan function. The function $\ope_{q}(z,u)$ is an $u$-deformed $q$-exponential function since when $u\mapsto1$, then $\ope_{q}(z,u)\mapsto\ope_{q}(z)$ and since $\ope_{q}(z)$ is a $q$-deformation of the exponential function $\ope^z$, then the function $\ope_{q}(z,u)$ defines a double deformation of the exponential function.

Associated to the deformed $q$-exponential function $\ope_{q}(z,u)$ we have the deformed homogeneous polynomials $\opR_{n}(x,y;u|q)$ defined by
\begin{equation}
    \opR_{n}(x,y;u|q)=\sum_{k=0}^{n}\qbinom{n}{k}_{q}u^{\binom{k}{2}}x^{n-k}y^{k},
\end{equation}
whose generating function is
\begin{equation}
    \ope_{q}(xt)\ope_{q}(yt,u)=\sum_{n=0}^{\infty}\opR_{n}(x,y;u|q)\frac{t^n}{(q;q)_{n}}.
\end{equation}
The polynomials $\opR_{n}(x,y;u|q)$ are a generalization of classical polynomials. For example:
\begin{itemize}
    \item If $x=1$, $y=x$, and $u=1$, then
    \begin{equation}
    \opR_{n}(1,x;1|q)=\oph_{n}(x|q),
\end{equation}
where $\oph_{n}(x|q)$ is the classical Rogers-Szeg\"o polynomial of degree $n$. The Rogers-Szeg\"o polynomials play an important role in the theory of the orthogonal polynomials \cite{Al} and physics \cite{jag}. 

\item If $u=1$, then
\begin{equation}
    \opR_{n}(x,y;1|q)=\mathrm{r}_{n}(x,y),
\end{equation}
where $\mathrm{r}_{n}(x,y)$ is a generalization of the Rogers-Szeg\"o polynomials given by Saad et al. \cite{saad}.

\item If $y\mapsto-y$ and $u=q$, then
\begin{equation}
    \opR_{n}(x,-y;q|q)=\opP_{n}(x,y),
\end{equation}
where $\opP_{n}(x,y)$ are the Cauchy polynomials
\begin{equation*}
    \opP_{n}(x,y)=(x-y)(x-qy)(x-q^2y)\cdots(x-q^{n-1}y).
\end{equation*}

\item If $x=1$, $y=qx$, and $u=q^2$, then
\begin{equation}
    \opR_{n}(1,qx;q^2|q)=(q;q)_{n}\opS_{n}(x;q)={}_{1}\phi_{1}\left(\begin{array}{c}
         q^{-n}\\
         0
    \end{array};q,-q^{n+1}x\right),
\end{equation}
where $\opS_{n}(x;q)$ are the Stieltjes-Wigert polynomials.
\end{itemize}
In this paper, we give basic properties of the polynomials $\opR_{n}$, their $q$-difference equation, representations, and recurrence relations. Five types of generating functions for the polynomials $\opR_{n}(x,y;u|q)$ are obtained
\begin{align*}
    &\sum_{n=0}^{\infty}v^{\binom{n}{2}}\opR_{n}(x,y;u|q)\frac{t^n}{(q;q)_{n}},\text{  (Deformed generating function)}\\
    &\sum_{n=0}^{\infty}\opR_{n}(x,y;u|q)\frac{(a;q)_{n}}{(q;q)_{n}}t^n,\text{  (Srivastava-Agarwal type)}\\
    &\sum_{n=0}^{\infty}\opR_{n}(x,y;u|q)\frac{(z;q)_{n}t^n}{(az;q)_{n}(q;q)_{n}},\\
    &\sum_{n=0}^{\infty}\opR_{n}(x,y;u|q)\opR_{n}(z,w;v|q)\frac{t^n}{(q;q)_{n}},\text{  (Mehler type)}\\
    &\sum_{n=0}^{\infty}\sum_{m=0}^{\infty}\opR_{n+m}(x,y;u|q)\frac{v^{\binom{n}{2}}w^{\binom{m}{2}}t^n}{(q;q)_{n}(q;q)_{m}},\text{  (Rogers type)}
\end{align*}
and many special cases are given. These generating functions are expressed in terms of the deformed basic hypergeometric series ${}_{r}\Phi_{s}$ defined by
\begin{align}
        &{}_{r}\Phi_{s}\left(
    \begin{array}{c}
         a_{1},a_{2},\ldots,a_{r} \\
         b_{1},\ldots,b_{s}
    \end{array}
    ;q,u,z
    \right)\nonumber\\
    &\hspace{3cm}=\sum_{n=0}^{\infty}u^{\binom{n}{2}}\frac{(a_{1},a_{2},\ldots,a_{r};q)_{n}}{(q,b_{1},b_{2},\ldots,b_{s};q)_{n}}\bigg[(-1)^nq^{\binom{n}{2}}\bigg]^{1+s-r}z^n.                
    \end{align}

The paper is divided into the following sections: Section 2 discusses the basic identities of the $q$-differential operator. In Section 3, the deformed basic hypergeometric series ${}_{r}\Phi_{s}$ is introduced. As we will note later, these series are closely related to the polynomials $\opR_{n}(x,y;u,v|q)$. In this section, we will give the equation in $q$-difference of the ${}_{2}\Phi_{1}$-series. Section 4 introduces the deformed homogeneous polynomials. Basic properties, $q$-difference equations, basic hypergeometric representations, and recurrence relations are given. Section 5 introduces the deformed $q$-exponential operator and provides some identities for this operator. In Section 6, generating functions for $\opR_{n}(x,y;u|q)$ are given: generalized $q$-binomial theorem, generalization of Heine's transformation formula, Mehler's, Srivastava-Agarwal, and Rogers type formulas. New transformation formulas for basic hypergeometric series are given.

\section{The q-differential operator}
The $q$-differential operator $D_{q}$ is defined by:
\begin{equation}\label{eqn_qdiff}
    D_{q}f(x)=\frac{f(x)-f(qx)}{x}
\end{equation}
and the Leibniz rule for $D_{q}$ is
\begin{equation}\label{eqn_leibniz}
    D_{q}^{n}\{f(x)g(x)\}=\sum_{k=0}^{n}q^{k(k-n)}\qbinom{n}{k}_{q}D_{q}^{k}\{f(x)\}D_{q}^{n-k}\{g(q^{k}x)\}.
\end{equation}
Then
\begin{equation*}
    D_{q}^nx^{k}=\frac{(q;q)_{k}}{(q;q)_{k-n}}x^{k-n}.
\end{equation*}
From Eq.~\eqref{eqn_qdiff}, we have identities for the deformed $q$-exponential function:
\begin{align}\label{eqn_kder_basic}
&\ope_{q}(z,u)-\ope_{q}(qz,u)-z\ope_{q}(uz,u)=0.\\
 &D_{q}^{k}\{\ope_{q}(ax,u)\}=a^{k}u^{\binom{k}{2}}\ope_{q}(au^{k}x,u).
\end{align}
According to the Leibniz formula and from Eqs. (\ref{eqn_leibniz}) and (\ref{eqn_kder_basic}),
\begin{align}\label{eqn_lei_prod_exp}
    &D_{q}^{n}\{\ope_{q}(ax,u)\ope_{q}(bx,v)\}\nonumber
    \\
    &\hspace{2cm}=\sum_{k=0}^{n}\qbinom{n}{k}_{q}u^{\binom{k}{2}}v^{\binom{n-k}{2}}a^kb^{n-k}
    \ope_{q}\left(au^{k}x,u\right)\ope_{q}\left(bq^{k}v^{n-k}x,v\right).
\end{align}
If $u=v=q$ in Eq.~\eqref{eqn_lei_prod_exp}, then
    \begin{align}\label{eqn_iden6}
        D_{q}^{n}\{(ax,bx;q)_{\infty}\}=q^{\binom{n}{2}}(ax,bq^nx;q)_{\infty}\sum_{k=0}^{n}\qbinom{n}{k}_{q}q^{k(k-n)}\frac{a^kb^{n-k}}{(ax;q)_{k}}.
    \end{align}
If $u=q$ and $v=1$ in Eq.~\eqref{eqn_lei_prod_exp}, then
    \begin{align}\label{eqn_iden7}
        D_{q}^{n}\left\{\frac{(-ax;q)_{\infty}}{(bx;q)_{\infty}}\right\}=\frac{(-ax;q)_{\infty}}{(bx;q)_{\infty}}\sum_{k=0}^{n}\qbinom{n}{k}_{q}q^{\binom{k}{2}}a^kb^{n-k}\frac{(bx;q)_{k}}{(-ax;q)_{k}}.
    \end{align}
If $u=v=1$ in Eq.~\eqref{eqn_lei_prod_exp}, then
    \begin{equation}\label{eqn_iden8}
        D_{q}^{n}\left\{\frac{1}{(ax,bx;q)_{\infty}}\right\}
        =\frac{1}{(ax,bx;q)_{\infty}}\sum_{k=0}^{n}\qbinom{n}{k}_{q}a^kb^{n-k}(bx;q)_{k}.
    \end{equation}

\section{Deformed basic hypergeometric series}

\begin{definition}
Set $0<\vert q\vert<1$ and take $u\in\C$. We define the $u$-deformed basic hypergeometric series ${}_{r}\Phi_{s}$ as
    \begin{align}
        &{}_{r}\Phi_{s}\left(
    \begin{array}{c}
         a_{1},a_{2},\ldots,a_{r} \\
         b_{1},\ldots,b_{s}
    \end{array}
    ;q,u,z
    \right)\nonumber\\
    &\hspace{3cm}=\sum_{n=0}^{\infty}u^{\binom{n}{2}}\frac{(a_{1},a_{2},\ldots,a_{r};q)_{n}}{(q,b_{1},b_{2},\ldots,b_{s};q)_{n}}\bigg[(-1)^nq^{\binom{n}{2}}\bigg]^{1+s-r}z^n.                
    \end{align}
\end{definition}
By letting that $u$ take values $1, q, q^2$ or $\sqrt{q}$, we obtain the following basic hypergeometric series:
\begin{itemize}
    \item If $u=1$, then ${}_{r}\Phi_{s}={}_{r}\phi_{s}$.
    \item If $u=q$, 
\begin{equation}
    {}_{r+1}\Phi_{r}\left(
    \begin{array}{c}
         a_{1},a_{2},\ldots,a_{r},0 \\
         b_{1},\ldots,b_{r}
    \end{array}
    ;q,q,z
    \right)={}_{r}\phi_{r}\left(
    \begin{array}{c}
         a_{1},a_{2},\ldots,a_{r} \\
         b_{1},b_{2},\ldots,b_{r}
    \end{array}
    ;q,-z
    \right)
\end{equation}
for all $z\in\C$.
\item If $u=q^2$ and mapping $z\mapsto qz$,
\begin{equation}
    {}_{r+1}\Phi_{r}\left(
    \begin{array}{c}
         a_{1},\ldots,a_{r+1} \\
         b_{1},\ldots,b_{r}
    \end{array}
    ;q,q^2,qz
    \right)={}_{r+1}\phi_{r+2}\left(
    \begin{array}{c}
         a_{1},\ldots,a_{r+1} \\
         b_{1},\ldots,b_{r},0,0
    \end{array}
    ;q,qz
    \right)
\end{equation}
for all $z\in\C$.
\item If $u=\sqrt{q}$,
\begin{align}
    &{}_{r+1}\Phi_{r}\left(
    \begin{array}{c}
         a_{1},\ldots,a_{r+1} \\
         b_{1},\ldots,b_{r}
    \end{array}
    ;q,\sqrt{q},z
    \right)\nonumber\\
    &\hspace{1cm}={}_{2r+2}\phi_{2r+2}\left(
    \begin{array}{ccc}
         \sqrt{a_{1}},-\sqrt{a_{1}}&,\ldots,&\sqrt{a_{r+1}},-\sqrt{a_{r+1}}\hspace{0.6cm}\\
         \sqrt{b_{1}},-\sqrt{b_{1}}&,\ldots,&\sqrt{b_{r}},-\sqrt{b_{r}},-\sqrt{q},0
    \end{array}
    ;\sqrt{q},-z
    \right)
\end{align}
for all $z\in\C$. 
\end{itemize}

Therefore, a deformed basic hypergeometric series generalizes the basic hypergeometric series. A series $\sum_{n=0}^{\infty}v_{n}$ is a deformed basic hypergeometric series if the quotient $v_{n+1}/v_{n}$ is a rational function of $q^n$ for a fixed base $q$ and if it is proportional to $u^n$. The most general form of the quotient is
\begin{equation*}
    \frac{v_{n+1}}{v_{n}}=u^n\frac{(1-a_{1}q^n)(1-a_{2}q^n)\cdots(1-a_{r}q^n)}{(1-q^{n+1})(1-b_{1}q^n)\cdots(1-b_{s}q^n)}(-q^n)^{1+s-r}z
\end{equation*}
normalizing $v_{0}=1$.
\begin{theorem}
Some convergence conditions for the ${}_{r}\Phi_{s}$-series are
\begin{itemize}
    \item $1+s-r\neq0$. If $0<\vert uq^{1+s-r}\vert<1$, then ${}_{r}\Phi_{s}$ is an entire function. If $\vert uq^{1+s-r}\vert=1$, then ${}_{r}\Phi_{s}$ converges for $\vert z\vert<1$. If $\vert uq^{1+s-r}\vert>1$, then ${}_{r}\Phi_{s}$ is divergent.
    \item $1+s-r=0$. If $0<\vert u\vert<1$, then ${}_{r}\Phi_{s}$ is an entire function. If $\vert u\vert=1$, then ${}_{r}\Phi_{s}$ converges for $\vert z\vert<1$. If $\vert u\vert>1$, then ${}_{r}\Phi_{s}$ is divergent.
\end{itemize}
\end{theorem}
The deformed $q$-exponential function has the following representation in $u$-deformed basic hypergeometric series
\begin{equation}
    \ope_{q}(z,u)={}_{1}\Phi_{0}\left(
    \begin{array}{c}
         0 \\
         -
    \end{array}
    ;q,u,z    
    \right).
\end{equation}
The deformed basic hypergeometric series ${}_{2}\Phi_{1}$ is
\begin{equation}\label{eqn_def_he}
    {}_{2}\Phi_{1}\left(
    \begin{array}{c}
         a,b \\
         c
    \end{array}
    ;q,u,z \right)=\sum_{n=0}^{\infty}u^{\binom{n}{2}}\frac{(a,b;q)_{n}}{(c;q)_{n}(q;q)_{n}}z^n.
\end{equation}
The $q$-difference operator applied to the ${}_{2}\Phi_{1}$-series:
\begin{equation}
    D_{q}^n{}_{2}\Phi_{1}\left(
    \begin{array}{c}
         a,b \\
         c
    \end{array}
    ;q,u,z    
    \right)=u^{\binom{n}{2}}\frac{(a,b;q)_{n}}{(c;q)_{n}}{}_{2}\Phi_{1}\left(
    \begin{array}{c}
         aq^n,bq^n \\
         cq^n
    \end{array}
    ;q,u,u^nz    
    \right)
\end{equation}
which can be checked directly. 
\begin{theorem}
The deformed basic hypergeometric series ${}_{2}\Phi_{1}$, Eq.~\eqref{eqn_def_he}, satisfies the $q$-difference equation
\begin{align}\label{eqn_qdiff_dbh}
    &czD_{q}^{2}f(z)-abqz^2D_{q}^2f(uz)+(1-c)D_{q}f(z)\nonumber\\
    &\hspace{1cm}+\big[(1-a)(1-b)-(1-abq)\big]zD_{q}f(uz)-(1-a)(1-b)f(uz)=0.
\end{align}
\end{theorem}
\begin{proof}
Suppose that $f(x)=\sum_{n=0}^{\infty}v_{n}z^n$ is the solution of Eq.~\eqref{eqn_qdiff_dbh}. Then
\begin{align*}
    &czD_{q}^{2}f(z)-abqz^2D_{q}^{2}f(uz)+(1-c)D_{q}f(z)\\
    &\hspace{1cm}+\big[(1-a)(1-b)-(1-abq)\big]zD_{q}f(uz)-(1-a)(1-b)f(uz)\\
    &\hspace{0.5cm}=c\sum_{n=0}^{\infty}v_{n}(1-q^n)(1-q^{n-1})z^{n-1}-abq\sum_{n=0}^{\infty}v_{n}(1-q^n)(1-q^{n-1})u^nz^{n}\\
    &\hspace{1cm}+(1-c)\sum_{n=0}^{\infty}v_{n}(1-q^n)z^{n-1}+\big[(1-a)(1-b)-(1-abq)\big]\sum_{n=0}^{\infty}v_{n}(1-q^n)u^nz^{n}\\
    &\hspace{1cm}-(1-a)(1-b)\sum_{n=0}^{\infty}v_{n}u^nz^n=0,
\end{align*}
Therefore, 
\begin{align*}
    &v_{n+1}\\
    &=\frac{u^n[abq(1-q^n)(1-q^{n-1})-((1-a)(1-b)-(1-abq))(1-q^n)+(1-a)(1-b)]}{c(1-q^{n+1})(1-q^n)+(1-c)(1-q^{n+1})}v_{n}\\
    &=u^n\frac{(1-aq^n)(1-bq^n)}{(1-q^{n+1})(1-cq^n)}v_{n}.
\end{align*}
Then
\begin{equation*}
    v_{n}=\prod_{k=0}^{n-1}u^k\frac{(a;q)_{n}(b;q)_{n}}{(q;q)_{n}(c;q)_{n}}=u^{\binom{n}{2}}\frac{(a;q)_{n}(b;q)_{n}}{(q;q)_{n}(c;q)_{n}},
\end{equation*}
normalizing $v_{0}=1$.
\end{proof}
Replacing $a,b,c$ with $q^a$, $q^b$, $q^c$ in Eq.~\eqref{eqn_qdiff_dbh} and then dividing by $(1-q)^2$, and taking formal limits, shows that Eq.~\eqref{eqn_qdiff_dbh} tends to the functional-differential equation
\begin{equation}\label{eqn_fdiff_hyp}
    zf^{\prime\prime}(z)-z^2f^{\prime\prime}(uz)+cf^{\prime}(z)-(a+b+1)zf^{\prime}(uz)-abf(uz)=0
\end{equation}
as $q\rightarrow1$. 
Equation~(\ref{eqn_fdiff_hyp}) has the solution
\begin{equation}
    {}_{2}\opF_{1}\left(
    \begin{array}{c}
         a,b \\
         c
    \end{array}
    ;u,z    
    \right)=\sum_{n=0}^{\infty}u^{\binom{n}{2}}\frac{(a)_{n}(b)_{n}}{(c)_{n}}\frac{z^{n}}{n!}.
\end{equation}
When $u=1$, we recover the classical hypergeometric function.

\section{Deformed homogeneous polynomials}

\subsection{Definition and basic properties}

\begin{definition}
We define the $(u,v)$-deformed homogeneous polynomial as
\begin{equation}
    \opR_{n}(x,y;u,v|q)=\sum_{k=0}^{n}\qbinom{n}{k}_{q}u^{\binom{n-k}{2}}v^{\binom{k}{2}}x^{n-k}y^{k}.
\end{equation}
\end{definition}
Some specializations of $\opR_{n}(x,y;u|q)$ are:
\begin{align}
    \opR_{n}(x,y;1,1|q)&=\opr_{n}(x,y),\label{eqn_poly1}\\
    \opR_{n}(1,x;q,q|q)&=\oph_{n}(x|q^{-1})=q^{\binom{n}{2}}\sum_{k=0}^{n}\qbinom{n}{k}_{q}q^{k(k-n)}x^k,\\
    \opR_{n}(1,-x;1,q|q)&=(x;q)_{n},\\
    \opR_{n}(1,qx;1,q^2|q)&=\opS^*_{n}(x;q)=(q;q)_{n}\mathrm{S}_{n}(x;q).
\end{align}
We define two new polynomials: The homogeneous Stieltjes-Wigert polynomials
\begin{equation}
    \opS_{n}(x,y;q)=\sum_{k=0}^{n}\qbinom{n}{k}_{q}q^{k^2}x^{n-k}y^{k}
\end{equation}
and the Exton polynomials 
\begin{align}
    \opE_{n}(x,y;q)&=\opR_{n}(x,y;1,\sqrt{q}|q)=\sum_{k=0}^{n}\qbinom{n}{k}_{q}\sqrt{q}^{\binom{k}{2}}x^{n-k}y^{k}\label{eqn_poly5}\\
    &=\sum_{k=0}^{n}\qbinom{n}{k}_{\sqrt{q}}\frac{(-\sqrt{q};\sqrt{q})_{n}}{(-\sqrt{q};\sqrt{q})_{k}(-\sqrt{q};\sqrt{q})_{n-k}}\sqrt{q}^{\binom{n}{2}}x^{n-k}y^{k}.
\end{align}
As we will see below, the Stieltjes-Wigert and Exton polynomials are related to the Rogers-Ramanujan and Exton $q$-exponential functions, respectively. As 
\begin{equation}
    \opR_{n}(x,y;u,v|q)=u^{\binom{n}{2}}\opR_{n}(x,u^{1-n}y;1,uv|q),
\end{equation}
then we will deal with $u$-deformed homogeneous polynomials
\begin{equation}\label{eqn_def_rs}
    \opR_{n}(x,y;u|q)\equiv\opR_{n}(x,y;1,u|q).
\end{equation}

\subsection{Recurrence relations}

\begin{theorem}
The polynomials $\opR_{n}(x,y;u|q)$ satisfy the following recursion relations
    \begin{align}
        \opR_{n+1}(x,y;u|q)&=x\opR_{n}(x,qy;u|q)+y\opR_{n}(x,uy;u|q),\\
        \opR_{n+1}(x,y;u|q)&=x\opR_{n}(x,y;u|q)+y\opR_{n}(qx,uy;u|q).
    \end{align}
\end{theorem}
\begin{proof}
By Eq. \eqref{eqn_def_rs},
\begin{align*}
    &x\opR_{n}(x,qy;u|q)+y\opR_{n}(x,uy;u|q)\\
    &\hspace{1cm}=\sum_{k=0}^{n}\qbinom{n}{k}_{q}q^ku^{\binom{k}{2}}x^{n+1-k}y^k+\sum_{k=0}^{n}\qbinom{n}{k}_{q}u^{\binom{k+1}{2}}x^{n-k}y^{k+1}\\
    &\hspace{1cm}=\sum_{k=0}^{n}\qbinom{n}{k}_{q}q^ku^{\binom{k}{2}}x^{n+1-k}y^k+\sum_{k=1}^{n+1}\qbinom{n}{k-1}_{q}u^{\binom{k}{2}}x^{n+1-k}y^{k}\\
    &\hspace{1cm}=x^{n+1}+u^{\binom{n+1}{2}}y^{n+1}+\sum_{k=1}^{n}\qbinom{n}{k}_{q}q^ku^{\binom{k}{2}}x^{n+1-k}y^k+\sum_{k=1}^{n}\qbinom{n}{k-1}_{q}u^{\binom{k}{2}}x^{n+1-k}y^{k}.
\end{align*}
Now, by using Eq.~\eqref{eqn_pascal} we have,
\begin{align*}
    &x\opR_{n}(x,qy;u|q)+y\opR_{n}(x,uy;u|q)\\
    &\hspace{1cm}=x^{n+1}+u^{\binom{n+1}{2}}y^{n+1}+\sum_{k=1}^{n}\bigg\{\qbinom{n}{k}_{q}q^k+\qbinom{n}{k-1}_{q}\bigg\}u^{\binom{k}{2}}x^{n+1-k}y^k\\
    &\hspace{1cm}=x^{n+1}+u^{\binom{n+1}{2}}y^{n+1}+\sum_{k=1}^{n}\qbinom{n+1}{k}_{q}u^{\binom{k}{2}}x^{n+1-k}y^k\\
    &\hspace{1cm}=\sum_{k=0}^{n+1}\qbinom{n+1}{k}_{q}u^{\binom{k}{2}}x^{n+1-k}y^k\\
    &\hspace{1cm}=\opR_{n+1}(x,y;u|q).
\end{align*}
The proof is reached.
\end{proof}
By iterating the above theorem, we have the following result.
\begin{theorem}
For $m\geq0$,
\begin{equation}
    \opR_{n+m}(x,y;u|q)=\sum_{k=0}^{m}\qbinom{m}{k}_{q}u^{\binom{k}{2}}x^{m-k}y^{k}\opR_{n}(x,q^{m-k}u^ky;u|q).
\end{equation}    
\end{theorem}

\subsection{Deformed q-exponential function as limit of deformed homogeneous polynomials}

\begin{theorem}
If $0<\vert q\vert<1$, then
    \begin{align}
        \lim_{n\rightarrow\infty}\opR_{n}(1,x;u|q)&=\ope_{q}(x,u).
    \end{align}
\end{theorem}
\begin{proof}
As
\begin{equation*}
    \lim_{n\rightarrow\infty}\qbinom{n}{k}_{q}=\frac{1}{(q;q)_{k}},
\end{equation*}
then
\begin{equation*}
    \lim_{n\rightarrow\infty}\opR_{n}(1,x;u|q)=\lim_{n\rightarrow\infty}\sum_{k=0}^{n}\qbinom{n}{k}_{q}u^{\binom{k}{2}}x^k=\sum_{k=0}^{\infty}u^{\binom{k}{2}}\frac{x^k}{(q;q)_{k}}=\ope_{q}(x,u).
\end{equation*}
\end{proof}

\subsection{Deformed basic hypergeometric representation}

\begin{theorem}
For all $n\geq0$ and $x\neq0$,
    \begin{equation}
    \opR_{n}(x,y;u|q)=x^n{}_{2}\Phi_{0}\left(
    \begin{array}{c}
         q^{-n},0 \\
         -
    \end{array}
    ;q,u,q^n y/x    
    \right).
\end{equation}
\end{theorem}
\begin{proof}
    \begin{align*}
        x^n{}_{2}\Phi_{0}\left(
    \begin{array}{c}
         q^{-n},0 \\
         -
    \end{array}
    ;q,u,q^n y/x  
    \right)
    &=x^n\sum_{k=0}^{\infty}u^{\binom{k}{2}}\frac{(q^{-n};q)_{k}}{(q;q)_{k}}(-1)^kq^{-\binom{k}{2}}(q^ny/x)^k\\
    &=\sum_{k=0}^{n}\qbinom{n}{k}_{q}u^{\binom{k}{2}}x^{n-k}y^k.
    \end{align*}
\end{proof}
The basic hypergeometric representation of polynomials Eqs.~\eqref{eqn_poly1}-(\ref{eqn_poly5}) are
\begin{align*}
    \opr_{n}(x,y)&=x^n{}_{2}\Phi_{0}\left(
    \begin{array}{c}
         q^{-n},0 \\
         -
    \end{array}
    ;q,1,q^n y/x    
    \right)=x^n{}_{2}\phi_{0}\left(
    \begin{array}{c}
         q^{-n},0 \\
         -
    \end{array}
    ;q,q^n y/x    
    \right).\\
    \oph_{n}(x|q^{-1})&=q^{\binom{n}{2}}{}_{2}\Phi_{0}\left(
    \begin{array}{c}
         q^{-n},0 \\
         -
    \end{array}
    ;q,q^2,qy    
    \right)=q^{\binom{n}{2}}{}_{1}\phi_{1}\left(
    \begin{array}{c}
         q^{-n},0 \\
         -
    \end{array}
    ;q, qy    
    \right).\\
    (x;q)_{n}&={}_{2}\Phi_{0}\left(
    \begin{array}{c}
         q^{-n},0 \\
         -
    \end{array}
    ;q,q,-q^{n}x    
    \right)={}_{2}\phi_{1}\left(
    \begin{array}{c}
         q^{-n},0 \\
         0
    \end{array}
    ;q,q^{n}x
    \right).\\
    \opS_{n}(x,y;q)&=x^n{}_{2}\Phi_{0}\left(
    \begin{array}{c}
         q^{-n},0 \\
         -
    \end{array}
    ;q,q^2,q^{n+1}y/x
    \right)=x^n{}_{1}\phi_{1}\left(
    \begin{array}{c}
         q^{-n} \\
         0
    \end{array}
    ;q,-q^{n+1}y/x    
    \right).\\
    \opE_{n}(x,y;q)&=x^{n
    }{}_{2}\Phi_{0}\left(
    \begin{array}{c}
         q^{-n},0 \\
         -
    \end{array}
    ;q,\sqrt{q},q^{n}y/x    
    \right)\\
    &=x^n{}_{4}\phi_{2}\left(
    \begin{array}{c}
         q^{-n/2},-q^{-n/2},q,0 \\
         \sqrt{q},-\sqrt{q}
    \end{array}
    ;\sqrt{q},q^{n}y/x    
    \right).
\end{align*}

\subsection{The q-difference equation}

\begin{theorem}
The polynomials $y(x)=\opR_{n}(1,x;u|q)$ are solutions of the proportional functional difference equation
\begin{equation}\label{eqn_gpfde}
    (D_{q}y)(u^{-1}x)+q^{n-1}x(D_{q}y)(q^{-1}x)=(1-q^n)y(x),
\end{equation}
with initial value $y(0)=1$. 
\end{theorem}
\begin{proof}
As
\begin{equation*}
    D_{q}\opR_{n}(1,x;u|q)=(1-q^n)\opR_{n-1}(1,ux;u|q),
\end{equation*}
then the left side of Eq.~\eqref{eqn_gpfde} is equal to
\begin{align*}
    &(D_{q}y)(u^{-1}x)+q^{n-1}x(D_{q}y)(q^{-1}x)\\
    &\hspace{1cm}=(1-q^n)\opR_{n-1}(1,x;u|q)+q^{n-1}(1-q^n)x\opR_{n-1}(1,uq^{-1}x;u|q)\\
    &\hspace{1cm}=(1-q^n)\bigg(\sum_{k=0}^{n-1}\qbinom{n-1}{k}_{q}u^{\binom{k}{2}}x^k+\sum_{k=1}^{n}\qbinom{n-1}{k-1}_{q}u^{\binom{k}{2}}q^{n-k}x^{k}\bigg)\\
    &\hspace{1cm}=(1-q^n)\sum_{k=0}^{n}\qbinom{n}{k}_{q}u^{\binom{k}{2}}x^{k}=(1-q^n)\opR_{n}(1,x;u|q).
\end{align*}
\end{proof}
Eq.~\eqref{eqn_gpfde} is equivalent to
\begin{equation}
    f(u^{-1}x)-f(qu^{-1}x)+q^{n}u^{-1}xf(q^{-1}x)-u^{-1}xf(x)=0.
\end{equation}

\section{The deformed q-exponential operator}

\begin{definition}\label{def_translation}
Set $u\in\C$. Define the $u$-deformed $q$-exponential operator $\opE(yD_{q}|u)$ by letting 
\begin{equation}
    \opE(yD_{q}|u)=\sum_{n=0}^{\infty}u^{\binom{n}{2}}\frac{(yD_{q})^{n}}{(q;q)_{n}}. 
\end{equation}
\end{definition}
The $u$-deformed $q$-exponential operator $\opE(yD_{q}|u)$ generalizes the operators given by Chen \cite{chen2}:
\begin{equation}
    \opT(bD_{q})=\sum_{n=0}^{\infty}\frac{(bD_{q})^n}{(q;q)_{n}}=\opE(bD_{q}|1),
\end{equation}
and by Saad et al. \cite{saad2}:
\begin{equation}
    \opR(bD_{q})=\sum_{n=0}^{\infty}(-1)^{n}\frac{(bD_{q})^nq^{\binom{n}{2}}}{(q;q)_{n}}=\opE(-bD_{q}|q).
\end{equation}
When $u=\sqrt{q}$, we define the Exton operator $\mathcal{E}(yD_{q})$ as
\begin{equation}
    \mathcal{E}(yD_{q})=\opE(yD_{q}|\sqrt{q})=\sum_{n=0}^{\infty}\sqrt{q}^{\binom{n}{2}}\frac{(yD_{q})^n}{(q;q)_{n}}={}_{1}\phi_{1}\left(\begin{array}{c}
         0\\
         -\sqrt{q}
    \end{array};\sqrt{q},-yD_{q}\right).
\end{equation}
When $u=q^2$ and by the mapping $y\mapsto qy$, we define the Rogers-Ramanujan operator $\mathcal{R}(yD_{q})$ as
\begin{equation}
    \mathcal{R}(yD_{q})=\opE(qyD_{q}|q^2)=\sum_{n=0}^{\infty}q^{n^2}\frac{(yD_{q})^n}{(q;q)_{n}}.
\end{equation}
The deformed homogeneous polynomials $\opR_{n}(x,y;u|q)$ can be represented by the deformed $q$-exponential operator $\opE(yD_{q}|u)$ as follows:
\begin{proposition}\label{prop_translation}
For all $u,v\in\C$
    \begin{align}
        \opE(yD_{q}|u)\big\{x^{n}\big\}&=\opR_{n}(x,y;u|q).
    \end{align}
\end{proposition}
\begin{proof}
By applying the deformed $q$-exponential operator, we get
    \begin{align*}
        \opE(yD_{q}|u)\big\{x^{n}\big\}
        &=\sum_{k=0}^{\infty}u^{\binom{k}{2}}\frac{y^{k}}{(q;q)_{k}}D_{q}^{k}\{x^{n}\}\\
        &=\sum_{k=0}^{\infty}\qbinom{n}{k}_{q}u^{\binom{k}{2}}x^{n-k}y^{k}\\
        &=\opR_{n}(x,y;u|q),
    \end{align*}
which implies the statement.
\end{proof}
Below, we have some $q$-operator identities.
\begin{theorem}\label{theo_opeE_exp}
If $u,v\in\C$, $u\neq0$, then
    \begin{equation}\label{eqn_opeE_exp}
        \opE(yD_{q}|u)\big\{\ope_{q}(ax,v)\big\}=\sum_{k=0}^{\infty}\left(uv\right)^{\binom{k}{2}}\frac{(ay)^k}{(q;q)_{k}}\ope_{q}\left(av^{k}x,v\right).
    \end{equation}
\end{theorem}
\begin{proof}
From Eq.~\eqref{eqn_kder_basic},
    \begin{align*}
        \opE(yD_{q}|u)\big\{\ope_{q}(ax,v)\big\}
        &=\sum_{k=0}^{\infty}\frac{u^{\binom{k}{2}}y^{k}}{(q;q)_{k}}D_{q}^k\{\ope_{q}(ax,v)\}\\
        &=\sum_{k=0}^{\infty}u^{\binom{k}{2}}\frac{y^k}{(q;q)_{k}}a^{k}v^{\binom{k}{2}}\ope_{q}\left(av^kx,v\right).
    \end{align*}
The proof is completed.
\end{proof}
If $v=1$ in Eq.~\eqref{eqn_opeE_exp}, then
\begin{align}\label{eqn_ope_EI}
    \opE(yD_{q}|u)\bigg\{\frac{1}{(ax;q)_{\infty}}\bigg\}&=\frac{\ope_{q}(ay,u)}{(ax;q)_{\infty}},
\end{align}
for $\vert ax\vert<1$. If $v=q$ in Eq.~\eqref{eqn_opeE_exp}, then
\begin{align}\label{eqn_ope_EII}
    \opE(yD_{q}|u)\{(ax;q)_{\infty}\}&=(ax;q)_{\infty}\cdot{}_{2}\Phi_{1}\left(
        \begin{array}{c}
             0,0 \\
             ax
        \end{array};
        q,qu, a y
        \right),
\end{align}
for $\vert ax\vert<1$.

\begin{theorem}\label{theo_opeE_prod_exp}
If $u,v,w\in\C$, $u\neq0$, then
    \begin{multline}
        \opE(yD_{q}|u)\{\ope_{q}(ax,v)\ope_{q}(bx,w)\}\\
        =\sum_{k=0}^{\infty}\sum_{n=0}^{\infty}(uv)^{\binom{k}{2}}(uw)^{\binom{n}{2}}\frac{(ay)^k}{(q;q)_{k}}\frac{(u^kby)^{n}}{(q;q)_{n}}\ope_{q}(av^kx,v)\ope_{q}(bq^{k}w^{n}x,w).
    \end{multline}
\end{theorem}
\begin{proof}
    \begin{align*}
        &\opE(yD_{q}|u)\{\ope_{q}(ax,v)\ope_{q}(bx,w)\}\\
        &\hspace{1cm}=\sum_{n=0}^{\infty}u^{\binom{n}{2}}\frac{y^n}{(q;q)_{n}}D_{q}^{n}\{\ope_{q}(ax,v)\ope_{q}(bx,w)\}\\
        &\hspace{1cm}=\sum_{n=0}^{\infty}u^{\binom{n}{2}}\frac{y^n}{(q;q)_{n}}\sum_{k=0}^{n}\qbinom{n}{k}_{q}a^kv^{\binom{k}{2}}b^{n-k}w^{\binom{n-k}{2}}\ope_{q}(av^kx,v)\ope_{q}(bq^{k}w^{n-k}x,w)\\
        &\hspace{1cm}=\sum_{k=0}^{\infty}(uv)^{\binom{k}{2}}\frac{(ay)^k}{(q;q)_{k}}\ope_{q}(av^kx,v)\sum_{n=0}^{\infty}\frac{(uw)^{\binom{n}{2}}(u^kby)^{n}}{(q;q)_{n}}\ope_{q}(bq^{k}w^{n}x,w)\\
        &\hspace{1cm}=\sum_{k=0}^{\infty}\sum_{n=0}^{\infty}(uv)^{\binom{k}{2}}(uw)^{\binom{n}{2}}\frac{(ay)^k}{(q;q)_{k}}\frac{(u^kby)^{n}}{(q;q)_{n}}\ope_{q}(av^kx,v)\ope_{q}(bq^{k}w^{n}x,w).
    \end{align*}
\end{proof}
We have some specializations of the above theorem:
\begin{corollary}\label{coro1}
If $v=w=1$, then
\begin{equation}
    \opE(yD_{q}|u)\left\{\frac{1}{(ax,bx;q)_{\infty}}\right\}
        =\frac{1}{(ax,bx;q)_{\infty}}\sum_{k=0}^{\infty}u^{\binom{k}{2}}\frac{(bx;q)_{k}(ay)^k}{(q;q)_{k}}\ope_{q}(u^kby,u).
\end{equation}
\end{corollary}

\begin{corollary}\label{coro2}
If $v=1$, $w=q$ and setting $b\mapsto-b$, then
\begin{equation}\label{eqn_coro2}
    \opE(yD_{q}|u)\left\{\frac{(bx;q)_{\infty}}{(ax;q)_{\infty}}\right\}\\
        =\frac{(bx;q)_{\infty}}{(ax;q)_{\infty}}\sum_{k=0}^{\infty}\frac{u^{\binom{k}{2}}(ay)^k}{(q;q)_{k}(bx;q)_{k}}{}_{1}\Phi_{1}\left(
        \begin{array}{c}
             0 \\
             bxq^k
        \end{array};
        q,u,u^kby
        \right).
\end{equation}
\end{corollary}

\begin{corollary}\label{coro3}
If $v=q$, $w=1$ and setting $a\mapsto-a$, then
\begin{equation}\label{eqn_coro2b}
    \opE(yD_{q}|u)\left\{\frac{(ax;q)_{\infty}}{(bx;q)_{\infty}}\right\}\\
        =\frac{(ax;q)_{\infty}}{(bx;q)_{\infty}}\sum_{k=0}^{\infty}\frac{(uq)^{\binom{k}{2}}(-ay)^k}{(q;q)_{k}(ax;q)_{k}}\ope_{q}(u^{k}by,u).
\end{equation}
\end{corollary}

\begin{theorem}\label{theo_oper_ax/bx}
    \begin{equation}
        \opE(yD_{q}|u)\left\{\frac{(ax;q)_{\infty}}{(bx;q)_{\infty}}\right\}=\frac{(ax;q)_{\infty}}{(bx;q)_{\infty}}{}_{2}\Phi_{1}\left(
        \begin{array}{c}
             a/b,0 \\
             ax
        \end{array};
        q,u,by
        \right).
    \end{equation}
\end{theorem}
\begin{proof}
    \begin{align*}
        \opE(yD_{q}|u)\left\{\frac{(ax;q)_{\infty}}{(bx;q)_{\infty}}\right\}        &=\sum_{n=0}^{\infty}u^{\binom{n}{2}}\frac{y^n}{(q;q)_{n}}D_{q}^n\left\{\frac{(ax;q)_{\infty}}{(bx;q)_{\infty}}\right\}\\
        &=\sum_{n=0}^{\infty}u^{\binom{n}{2}}\frac{y^n}{(q;q)_{n}}D_{q}^n\left\{{}_{1}\phi_{0}\left(
    \begin{array}{c}
         a/b\\
         -
    \end{array}
    ;q,bx
    \right)\right\}\\
    &=\sum_{n=0}^{\infty}u^{\binom{n}{2}}\frac{(by)^n}{(q;q)_{n}}(a/b;q)_{n}\cdot{}_{1}\phi_{0}\left(
    \begin{array}{c}
         aq^n/b\\
         -
    \end{array}
    ;q,bx
    \right)\\
    &=\sum_{n=0}^{\infty}u^{\binom{n}{2}}\frac{(a/b;q)_{n}}{(q;q)_{n}}(by)^{n}\frac{(aq^nx;q)_{\infty}}{(bx;q)_{\infty}}\\
    &=\frac{(ax;q)_{\infty}}{(bx;q)_{\infty}}{}_{2}\Phi_{1}\left(
    \begin{array}{c}
         a/b,0\\
         ax
    \end{array}
    ;q,u, by
    \right).
    \end{align*}
\end{proof}

\begin{corollary}\label{coro4}
If $v=w=q$ and setting $a\mapsto-a$ and $b\mapsto-b$, then
    \begin{multline}
        \opE(yD_{q}|u)\{(ax,bx;q)_{\infty}\}\\
        =(ax,bx;q)_{\infty}\sum_{k=0}^{\infty}\frac{(qu)^{\binom{k}{2}}(-ay)^k}{(ax;q)_{k}(bx;q)_{k}(q;q)_{k}}{}_{1}\Phi_{1}\left(
        \begin{array}{c}
             0 \\
             bxq^k
        \end{array};
        q,u,u^kby
        \right).
    \end{multline}
\end{corollary}
\section{Generating functions for \texorpdfstring{$\opR_{n}(x,y;u|q)$}{r\_n(x,y;u|q)}}

\subsection{Generalized q-binomial theorem}

\begin{theorem}\label{theo_v_gen_func}
    \begin{equation}
        \sum_{n=0}^{\infty}v^{\binom{n}{2}}\opR_{n}(x,y;u|q)\frac{z^n}{(q;q)_{n}}=\sum_{k=0}^{\infty}\left(uv\right)^{\binom{k}{2}}\frac{(yz)^k}{(q;q)_{k}}\ope_{q}\left(v^{k}xz,v\right).
    \end{equation}
\end{theorem}
\begin{proof}
From Theorem \ref{theo_opeE_exp}
    \begin{align*}
        \sum_{n=0}^{\infty}v^{\binom{n}{2}}\opR_{n}(x,y;u|q)\frac{z^n}{(q;q)_{n}}&=\sum_{n=0}^{\infty}v^{\binom{n}{2}}\opE(yD_{q}|u)\left\{\frac{(xz)^n}{(q;q)_{n}}\right\}\\
        &=\opE(yD_{q}|u)\left\{\sum_{n=0}^{\infty}v^{\binom{n}{2}}\frac{(xz)^n}{(q;q)_{n}}\right\}\\
        &=\opE(yD_{q}|u)\left\{\ope_{q}(xz,v)\right\}\\
        &=\sum_{k=0}^{\infty}\left(uv\right)^{\binom{k}{2}}\frac{(yz)^k}{(q;q)_{k}}\ope_{q}\left(v^{k}xz,v\right).
    \end{align*}
\end{proof}

\begin{corollary}\label{coro5}
Set $\vert yz\vert<1$ and $\vert v\vert<1$. Then
    \begin{equation}
        \sum_{n=0}^{\infty}v^{\binom{n}{2}}\opR_{n}(x,y;v^{-1}|q)\frac{z^n}{(q;q)_{n}}=\sum_{n=0}^{\infty}v^{\binom{n}{2}}\frac{(xz)^{n}}{(q;q)_{n}}\frac{1}{(v^nyz;q)_{\infty}}.
    \end{equation}
\end{corollary}
\begin{proof}
    \begin{align*}
        \sum_{n=0}^{\infty}v^{\binom{n}{2}}\opR_{n}(x,y;v^{-1}|q)\frac{z^n}{(q;q)_{n}}&=\sum_{k=0}^{\infty}\frac{(yz)^k}{(q;q)_{k}}\ope_{q}\left(v^{k}xz,v\right)\\
        &=\sum_{k=0}^{\infty}\frac{(yz)^k}{(q;q)_{k}}\sum_{n=0}^{\infty}v^{\binom{n}{2}}\frac{(v^{k}xz)^n}{(q;q)_{n}}\\
        &=\sum_{n=0}^{\infty}v^{\binom{n}{2}}\frac{(xz)^{n}}{(q;q)_{n}}\sum_{k=0}^{\infty}\frac{(v^nyz)^k}{(q;q)_{k}}\\
        &=\sum_{n=0}^{\infty}v^{\binom{n}{2}}\frac{(xz)^{n}}{(q;q)_{n}}\frac{1}{(v^nyz;q)_{\infty}}.
    \end{align*}
\end{proof}
From Corollary \ref{coro5} we have the following results:
\begin{align}\label{eqn_qRq-}
    \sum_{n=0}^{\infty}q^{\binom{n}{2}}\opR_{n}(x,y;q^{-1}|q)\frac{z^n}{(q;q)_{n}}&=\frac{1}{(yz;q)_{\infty}}{}_{1}\phi_{1}\left(
        \begin{array}{c}
             yz \\
             0
        \end{array};
        q,-xz
        \right).
\end{align}

\begin{corollary}
Set $v=q^2$, $u=q^{-2}$, and $z\mapsto qz$ in Theorem \ref{theo_v_gen_func}. Then 
\begin{multline}
    \sum_{n=0}^{\infty}q^{n^2}\opR_{n}(x,y;q^{-2}|q)\frac{z^n}{(q;q)_{n}}\\
    =\frac{1}{(yz;q)_{\infty}}{}_{4}\phi_{5}\left(
        \begin{array}{c}
             \sqrt{yz},-\sqrt{yz},\sqrt{yzq},-\sqrt{yzq} \\
             0,0,0,0,0
        \end{array};
        q,q^2xz
        \right).
\end{multline}
\end{corollary}
\begin{proof}
Set $v=q^2$, $u=q^{-2}$, and by mapping $z\mapsto qz$, then
    \begin{align*}
        &\sum_{n=0}^{\infty}q^{n^2}\opR_{n}(x,y;q^{-2}|q)\frac{z^n}{(q;q)_{n}}\\
        &=\sum_{m=0}^{\infty}q^{m^2}\frac{(xz)^m}{(q;q)_{m}}
        \frac{1}{(q^{2m}yz;q)_{\infty}}\\
        &=\frac{1}{(yz;q)_{\infty}}\sum_{m=0}^{\infty}q^{m^2}\frac{(xz)^m(yz;q)_{2m}}{(q;q)_{m}}\\
        &=\frac{1}{(yz;q)_{\infty}}\sum_{m=0}^{\infty}q^{m^2}\frac{(xz)^m(yz;q^2)_{m}(yzq;q^2)_{m}}{(q;q)_{m}}\\
        &=\frac{1}{(yz;q)_{\infty}}\sum_{m=0}^{\infty}q^{m^2}\frac{(xz)^m(\sqrt{yz};q)_{m}(-\sqrt{yz};q)_{m}(\sqrt{yzq};q)_{m}(-\sqrt{yzq};q)_{m}}{(q;q)_{m}}\\
        &=\frac{1}{(yz;q)_{\infty}}{}_{4}\phi_{5}\left(
        \begin{array}{c}
             \sqrt{yz},-\sqrt{yz},\sqrt{yzq},-\sqrt{yzq} \\
             0,0,0,0,0
        \end{array};
        q,qxz
        \right).
    \end{align*}
\end{proof}

\begin{corollary}
Set $v=\sqrt{q}$ and $u=\sqrt{q^{-1}}$ in Theorem \ref{theo_v_gen_func}. Then
\begin{align}
    &\sum_{n=0}^{\infty}\sqrt{q}^{\binom{n}{2}}\opR_{n}(x,y;\sqrt{q^{-1}}|q)\frac{z^n}{(q;q)_{n}}\nonumber\\
    &\hspace{1cm}=\frac{1}{(yz;q)_{\infty}}{}_{3}\phi_{4}\left(
        \begin{array}{c}
            0,0,0\\
             \sqrt{q},-\sqrt{q},-q,yz
        \end{array};
        q,\sqrt{q}x^2z^2
        \right)\nonumber\\
        &\hspace{2cm}+\frac{xz}{(1-q)(\sqrt{q}yz;q)_{\infty}}{}_{3}\phi_{4}\left(
        \begin{array}{c}
            0,0,0\\
             q\sqrt{q},-q\sqrt{q},-q,\sqrt{q}yz
        \end{array};
        q,q\sqrt{q}x^2z^2
        \right).
\end{align}
\end{corollary}
\begin{proof}
Set $v=\sqrt{q}$, $u=\sqrt{q^{-1}}$. Then
    \begin{align*}
        &\sum_{n=0}^{\infty}\sqrt{q}^{\binom{n}{2}}\opR_{n}(x,y;\sqrt{q^{-1}}|q)\frac{z^n}{(q;q)_{n}}\\
        &=\sum_{m=0}^{\infty}\sqrt{q}^{\binom{m}{2}}\frac{(xz)^m}{(q;q)_{m}}\frac{1}{(q^{m/2}yz;q)_{\infty}}\\
        &=\sum_{m=0}^{\infty}\sqrt{q}^{\binom{2m}{2}}\frac{(xz)^{2m}}{(q;q)_{2m}}\frac{1}{(q^{m}yz;q)_{\infty}}+\sum_{m=0}^{\infty}\sqrt{q}^{\binom{2m+1}{2}}\frac{(xz)^{2m+1}}{(q;q)_{2m+1}}\frac{1}{(q^{m+1/2}yz;q)_{\infty}}\\
        &=\frac{1}{(yz;q)_{\infty}}\sum_{m=0}^{\infty}\sqrt{q}^{\binom{2m}{2}}\frac{(xz)^{2m}}{(q;q)_{2m}}(yz;q)_{m}\\
        &\hspace{3cm}+\frac{1}{(\sqrt{q}yz;q)_{\infty}}\sum_{m=0}^{\infty}\sqrt{q}^{\binom{2m+1}{2}}\frac{(xz)^{2m+1}}{(q;q)_{2m+1}}(\sqrt{q}yz;q)_{m}.
    \end{align*}
Now by using the identities Eq.~\eqref{eqn_iden3} and Eq.~\eqref{eqn_iden4},
\begin{align*}
    &\sum_{n=0}^{\infty}\sqrt{q}^{\binom{n}{2}}\opR_{n}(x,y;\sqrt{q^{-1}}|q)\frac{z^n}{(q;q)_{n}}\\
    &=\frac{1}{(yz;q)_{\infty}}\sum_{m=0}^{\infty}\sqrt{q}^{\binom{2m}{2}}\frac{(yz;q)_{m}(xz)^{2m}}{(\sqrt{q};q)_{m}(-\sqrt{q};q)_{m}(-q;q)_{m}(q;q)_{m}}\\
        &\hspace{1cm}+\frac{1}{(1-q)(\sqrt{q}yz;q)_{\infty}}\sum_{m=0}^{\infty}\sqrt{q}^{\binom{2m+1}{2}}\frac{(\sqrt{q}yz;q)_{m}(xz)^{2m+1}}{(-q;q)_{m}(q;q)_{m}(-q\sqrt{q};q)_{m}(q\sqrt{q};q)_{m}}\\
        &=\frac{1}{(yz;q)_{\infty}}{}_{2}\phi_{3}\left(
        \begin{array}{c}
            yz,0\\
             \sqrt{q},-\sqrt{q},-q
        \end{array};
        q,\sqrt{q}x^2z^2
        \right)\\
        &\hspace{3cm}+\frac{xz}{(1-q)(\sqrt{q}yz;q)_{\infty}}{}_{2}\phi_{3}\left(
        \begin{array}{c}
            \sqrt{q}yz,0\\
             q\sqrt{q},-q\sqrt{q},-q
        \end{array};
        q,q\sqrt{q}x^2z^2
        \right).
\end{align*}
\end{proof}

\begin{corollary}[Generalized $q$-binomial theorem]\label{coro6}
If $v=1$ in Theorem \ref{theo_v_gen_func}, then
\begin{equation}\label{eqn_cauchy_gen}
    \sum_{n=0}^{\infty}\opR_{n}(x,y;u|q)\frac{t^n}{(q;q)_{n}}=\frac{\ope_{q}(yt,u)}{(xt;q)_{\infty}}.
\end{equation}    
\end{corollary}
Eq.~\eqref{eqn_cauchy_gen} is a generalization of the $q$-binomial theorem, and Eq.~\eqref{eqn_qbin_the} and Theorem \ref{theo_v_gen_func} provide a very straightforward proof for this. 
On the other hand,
\begin{itemize}
    \item 
    If $x=1$, $y=x$, and $u=1$ in Eq.~\eqref{eqn_cauchy_gen}, then 
\begin{equation*}
    \sum_{n=0}^{\infty}\oph_{n}(x|q)\frac{t^n}{(q;q)_{n}}=\frac{1}{(t,x t;q)_{\infty}},\ \max\{\vert t\vert,\vert xt\vert\}<1.
\end{equation*}
\item 
If $y\mapsto-y$ and $u=q$ in Eq.~\eqref{eqn_cauchy_gen}, then
\begin{equation*}
    \sum_{n=0}^{\infty}\opP_{n}(x,y)\frac{t^n}{(q;q)_{n}}=\frac{(yt;q)_{\infty}}{(x t;q)_{\infty}},\ \vert xt\vert<1.
\end{equation*}
\item 
If $u=q^2$, $y\mapsto qy$, and $y\neq0$ in Eq.~\eqref{eqn_cauchy_gen}, and if $yz=1$, then
\begin{equation}
    \sum_{n=0}^{\infty}\opS_{n}(x,y^{-1};q)\frac{y^n}{(q;q)_{n}}=\frac{\mathcal{R}_{q}(1)}{(xy^{-1};q)_{\infty}}=\frac{1}{(xy^{-1};q)_{\infty}(q;q^5)_{\infty}(q^4;q^5)_{\infty}}.
\end{equation}
If $yz=q$, then
\begin{equation}
    \sum_{n=0}^{\infty}\opS_{n}(x,qy^{-1};q)\frac{y^n}{(q;q)_{n}}=\frac{\mathcal{R}_{q}(q)}{(xy^{-1};q)_{\infty}}=\frac{1}{(xy^{-1};q)_{\infty}(q^2;q^5)_{\infty}(q^3;q^5)_{\infty}}.
\end{equation}
$\mathcal{R}_{q}(1)$ and $\mathcal{R}_{q}(q)$ are the Rogers-Ramanujan functions \cite{ramanujan,rogers}
\begin{equation}
    \mathcal{R}_{q}(1)=\frac{1}{(q;q^5)_{\infty}(q^4;q^5)_{\infty}
    }\text{  and  }\mathcal{R}_{q}(q)=\frac{1}{(q^2;q^5)_{\infty}(q^3;q^5)_{\infty}
    }.
\end{equation}
\item If $u=\sqrt{q}$ in Eq.~\eqref{eqn_cauchy_gen}. Then
\begin{equation}
    \sum_{n=0}^{\infty}\opE_{n}(x,y;q)\frac{t^n}{(q;q)_{n}}=\frac{\mathcal{E}_{q}(yt)}{(xt;q)_{\infty}}.
\end{equation}
\end{itemize}

\begin{corollary}\label{coro7}
If $v=q$ in Theorem \ref{theo_v_gen_func}, then
\begin{equation}
    \sum_{n=0}^{\infty}(-1)^nq^{\binom{n}{2}}\opR_{n}(x,y;u|q)\frac{z^n}{(q;q)_{n}}=(xz;q)_{\infty}{}_{1}\Phi_{1}\left(
        \begin{array}{c}
             0 \\
             xz
        \end{array};
        q,u,yz
        \right).
\end{equation}
\end{corollary}
\begin{proof}
\begin{align*}
    \sum_{n=0}^{\infty}(-1)^nq^{\binom{n}{2}}\opR_{n}(x,y;u|q)\frac{z^n}{(q;q)_{n}}&=\sum_{k=0}^{\infty}\left(uq\right)^{\binom{k}{2}}\frac{(-yz)^k}{(q;q)_{k}}\ope_{q}\left(-q^{k}xz,q\right)\\
    &=(xz;q)_{\infty}\sum_{k=0}^{\infty}\left(uq\right)^{\binom{k}{2}}\frac{(-yz)^k}{(xz;q)_{k}(q;q)_{k}}\\
    &=(xz;q)_{\infty}\cdot{}_{1}\Phi_{1}\left(
        \begin{array}{c}
             0 \\
             xz
        \end{array};
        q,u,yz
        \right).
\end{align*}    
\end{proof}
Some specializations of Corollary \ref{coro7} are
\begin{itemize}
    \item If $u=1$, then
\begin{equation}
    \sum_{n=0}^{\infty}(-1)^nq^{\binom{n}{2}}\opR_{n}(x,y)\frac{z^n}{(q;q)_{n}}=(xz;q)_{\infty}\cdot{}_{1}\phi_{1}\left(
        \begin{array}{c}
             0 \\
             xz
        \end{array};
        q,yz
        \right).
\end{equation}
\item If $u=q$, then
\begin{equation}
    \sum_{n=0}^{\infty}(-1)^nq^{\binom{n}{2}}\opP_{n}(x,y)\frac{z^n}{(q;q)_{n}}=(xz;q)_{\infty}\cdot{}_{0}\phi_{1}\left(
        \begin{array}{c}
             - \\
             xz
        \end{array};
        q,-yz
        \right).
\end{equation}
\item If $u=q^2$, then
\begin{equation}
    \sum_{n=0}^{\infty}(-1)^nq^{\binom{n}{2}}\opS_{n}(x,y;q)\frac{z^n}{(q;q)_{n}}=(xz;q)_{\infty}\cdot{}_{0}\phi_{2}\left(
        \begin{array}{c}
             0 \\
             xz,0
        \end{array};
        q,yz
        \right).
\end{equation}
\item If $u=\sqrt{q}$, then
\begin{equation}
    \sum_{n=0}^{\infty}(-1)^nq^{\binom{n}{2}}\opE_{n}(x,y;q)\frac{z^n}{(q;q)_{n}}=(xz;q)_{\infty}\cdot{}_{1}\phi_{3}\left(
        \begin{array}{c}
             0 \\
             \sqrt{xz},-\sqrt{xz},-\sqrt{q}
        \end{array};
        \sqrt{q},yz
        \right).
\end{equation}
\end{itemize}

\subsection{Generalization of Heine's transformation formula}

Heine's transformation formula is
\begin{equation}\label{eqn_heine}
    {}_{2}\phi_{1}\left(
        \begin{array}{c}
             a,b \\
             c
        \end{array};
        q,z
        \right)=\frac{(b,az;q)_{\infty}}{(c,z;q)_{\infty}}{}_{2}\phi_{1}\left(
        \begin{array}{c}
             c/b,z \\
             az
        \end{array};
        q,b
        \right)
\end{equation}
where $\vert z\vert<1$ and $\vert b\vert<1$. In this section, we investigate a generalization of Heine's transformation formula via the generating function
\begin{equation*}
    \sum_{n=0}^{\infty}\opR_{n}(x,y;u|q)\frac{(z;q)_{n}}{(az;q)_{n}(q;q)_{n}}t^n.
\end{equation*}

\begin{definition}
We define the following representation for the deformed homogeneous polynomials $\opR_{n}(x,y;u|q)$ based on the basic hypergeometric series
    \begin{multline}
        {}_{r}\rbhs_{s}\left(x,y;u;
        \begin{array}{c}
             a_{1},\ldots,a_{r} \\
             b_{1},\ldots,b_{s}
        \end{array};
        q,t
        \right)\\
        =\sum_{n=0}^{\infty}\opR_{n}(x,y;u|q)\frac{(a_{1},\ldots,a_{r};q)_{n}}{(q;q)_{n}(b_{1},\ldots,b_{s};q)_{n}}\Big[(-1)^{n}q^{\binom{n}{2}}\Big]^{1+s-r}t^n.
    \end{multline}
For $u=1,q,q^2,\sqrt{q}$, we have, respectively:
\begin{multline}
        {}_{r}\opR_{s}\left(x,y;
        \begin{array}{c}
             a_{1},\ldots,a_{r} \\
             b_{1},\ldots,b_{s}
        \end{array};
        q,t
        \right)\\
        =\sum_{n=0}^{\infty}\opr_{n}(x,y)\frac{(a_{1},\ldots,a_{r};q)_{n}}{(q;q)_{n}(b_{1},\ldots,b_{s};q)_{n}}\Big[(-1)^{n}q^{\binom{n}{2}}\Big]^{1+s-r}t^n.
    \end{multline}
    \begin{multline}
        {}_{r}\opP_{s}\left(x,y;
        \begin{array}{c}
             a_{1},\ldots,a_{r} \\
             b_{1},\ldots,b_{s}
        \end{array};
        q,t
        \right)\\
        =\sum_{n=0}^{\infty}\opP_{n}(x,y)\frac{(a_{1},\ldots,a_{r};q)_{n}}{(q;q)_{n}(b_{1},\ldots,b_{s};q)_{n}}\Big[(-1)^{n}q^{\binom{n}{2}}\Big]^{1+s-r}t^n.
    \end{multline}
    \begin{multline}
        {}_{r}\opS_{s}\left(x,y;
        \begin{array}{c}
             a_{1},\ldots,a_{r} \\
             b_{1},\ldots,b_{s}
        \end{array};
        q,t
        \right)\\
        =\sum_{n=0}^{\infty}\opS_{n}(x,y;q)\frac{(a_{1},\ldots,a_{r};q)_{n}}{(q;q)_{n}(b_{1},\ldots,b_{s};q)_{n}}\Big[(-1)^{n}q^{\binom{n}{2}}\Big]^{1+s-r}t^n.
    \end{multline}
    \begin{multline}
        {}_{r}\opE_{s}\left(x,y;
        \begin{array}{c}
             a_{1},\ldots,a_{r} \\
             b_{1},\ldots,b_{s}
        \end{array};
        q,t
        \right)\\
        =\sum_{n=0}^{\infty}\opE_{n}(x,y;q)\frac{(a_{1},\ldots,a_{r};q)_{n}}{(q;q)_{n}(b_{1},\ldots,b_{s};q)_{n}}\Big[(-1)^{n}q^{\binom{n}{2}}\Big]^{1+s-r}t^n.
    \end{multline}
\end{definition}
From Eqs.~\eqref{eqn_qRq-},
\begin{align*}
    {}_{0}\rbhs_{0}\left(x,y;q^{-1};
        \begin{array}{c}
             -\\
             -
        \end{array};
        q,-z
        \right)&=\frac{1}{(yz;q)_{\infty}}{}_{1}\phi_{1}\left(
        \begin{array}{c}
             yz \\
             0
        \end{array};
        q,-xz
        \right).\\
        {}_{1}\rbhs_{0}\left(x,y;u;
        \begin{array}{c}
             0\\
             -
        \end{array};
        q,z
        \right)&=\frac{\ope_{q}(yt,u)}{(xt;q)_{\infty}}.\\
        {}_{0}\rbhs_{0}\left(x,y;u;
        \begin{array}{c}
             -\\
             -
        \end{array};
        q,z
        \right)&=(xz;q)_{\infty}{}_{1}\Phi_{1}\left(
        \begin{array}{c}
             0 \\
             xz
        \end{array};
        q,u,yz
        \right).
\end{align*}

\begin{theorem}\label{theo_repre_DHP}
We get the following representation for the deformed homogeneous polynomials $\opR_{n}(x,y;u|q)$
    \begin{equation}
        {}_{1}\rbhs_{1}\left(x,y;u;
        \begin{array}{c}
             z \\
             az
        \end{array};
        q,b
        \right)=\frac{(z;q)_{\infty}}{(xb,az;q)_{\infty}}\sum_{n=0}^{\infty}\ope_{q}(ybq^n,u)\frac{(a;q)_{n}(xb;q)_{n}}{(q;q)_{n}}z^n,
    \end{equation}
where $\vert b\vert<1$ and $\vert z\vert<1$.
\end{theorem}
\begin{proof}
From generalized $q$-binomial theorem
\begin{equation}
    \frac{\ope_{q}(ybq^n,u)}{(xbq^n;q)_{\infty}}=\sum_{m=0}^{\infty}\opR_{m}(x,y;u|q)\frac{(bq^{n})^m}{(q;q)_{m}}.
\end{equation}
Hence,
\begin{align*}
    \sum_{n=0}^{\infty}\ope_{q}(ybq^n,u)\frac{(a;q)_{n}(xb;q)_{n}}{(q;q)_{n}}z^n&=(xb;q)_{\infty}\sum_{n=0}^{\infty}\frac{(a;q)_{n}}{(q;q)_{n}}\frac{\ope_{q}(ybq^n,u)}{(xbq^n;q)_{\infty}}z^n\\
    &=(xb;q)_{\infty}\sum_{n=0}^{\infty}\frac{(a;q)_{n}}{(q;q)_{n}}\sum_{m=0}^{\infty}\opR_{m}(x,y;u|q)\frac{(bq^{n})^m}{(q;q)_{m}}z^n\\
    &=(xb;q)_{\infty}\sum_{m=0}^{\infty}\opR_{m}(x,y;u|q)\frac{b^m}{(q;q)_{m}}\sum_{n=0}^{\infty}\frac{(a;q)_{n}}{(q;q)_{n}}(zq^m)^n\\
    &=(xb;q)_{\infty}\sum_{m=0}^{\infty}\opR_{m}(x,y;u|q)\frac{b^m}{(q;q)_{m}}\frac{(azq^m;q)_{\infty}}{(zq^m;q)_{\infty}}\\
    &=\frac{(az,xb;q)_{\infty}}{(z;q)_{\infty}}\sum_{m=0}^{\infty}\opR_{m}(x,y;u|q)\frac{(z;q)_{m}b^m}{(az;q)_{m}(q;q)_{m}}\\
\end{align*}
\end{proof}

\begin{corollary}
From  Theorem \ref{theo_repre_DHP} we have the following representation for the generalized Rogers-Szeg\"o polynomials
    \begin{equation}
        {}_{2}\opR_{1}\left(x,y;
        \begin{array}{c}
             z,0 \\
             az
        \end{array};
        q,b
        \right)=\frac{(z;q)_{\infty}}{(az,bx,by;q)_{\infty}}{}_{3}\phi_{2}\left(
        \begin{array}{c}
             a,bx,by \\
             0,0
        \end{array};
        q, z
        \right).
    \end{equation}
\end{corollary}
\begin{proof}
Set $u=1$ in Theorem \ref{theo_repre_DHP}. Then
    \begin{align*}
        \sum_{n=0}^{\infty}\opr_{n}(x,y)\frac{(z;q)_{n}}{(az;q)_{n}(q;q)_{n}}b^n&=\frac{(z;q)_{\infty}}{(bx,by,az;q)_{\infty}}\sum_{n=0}^{\infty}\frac{(a;q)_{n}(bx;q)_{n}(by;q)_{n}}{(q;q)_{n}}z^n\\
        &=\frac{(z;q)_{\infty}}{(az,bx,by;q)_{\infty}}{}_{3}\phi_{2}\left(
        \begin{array}{c}
             a,bx,by \\
             0,0
        \end{array};
        q, z
        \right).
    \end{align*}
\end{proof}

\begin{corollary}\label{coro11}
From  Theorem \ref{theo_repre_DHP}, we have the following representation for the Cauchy polynomials
    \begin{equation}
        {}_{2}\opP_{1}\left(x,y;
        \begin{array}{c}
             z,0 \\
             az
        \end{array};
        q,b
        \right)=\frac{(z,by;q)_{\infty}}{(az,bx;q)_{\infty}}{}_{2}\phi_{1}\left(
        \begin{array}{c}
             a,bx\\
             by
        \end{array};
        q, z
        \right).
    \end{equation}
\end{corollary}
\begin{proof}
Set $u=q$ and $y\mapsto-y$ in Theorem \ref{theo_repre_DHP}. Then
    \begin{align*}
        \sum_{n=0}^{\infty}\opP_{n}(x,y)\frac{(z;q)_{n}}{(az;q)_{n}(q;q)_{n}}b^n&=\frac{(z;q)_{\infty}}{(bx,az;q)_{\infty}}\sum_{n=0}^{\infty}(byq^n;q)_{\infty}\frac{(a;q)_{n}(bx;q)_{n}}{(q;q)_{n}}z^n\\
        &=\frac{(z,by;q)_{\infty}}{(bx,az;q)_{\infty}}\sum_{n=0}^{\infty}\frac{(a;q)_{n}(bx;q)_{n}}{(by;q)_{n}(q;q)_{n}}z^n\\
        &=\frac{(z,by;q)_{\infty}}{(bx,az;q)_{\infty}}{}_{2}\phi_{1}\left(
        \begin{array}{c}
             a,bx\\
             by
        \end{array};
        q, z
        \right).
    \end{align*}
\end{proof}

\begin{corollary}
From  Theorem \ref{theo_repre_DHP} we have the following representation for the homogeneous Stieltjes-Wigert polynomials
    \begin{equation}
        {}_{2}\opS_{1}\left(x,y;
        \begin{array}{c}
             z,0 \\
             az
        \end{array};
        q,b
        \right)=\frac{(z;q)_{\infty}}{(xb,az;q)_{\infty}}\sum_{n=0}^{\infty}\RR_{q}(ybq^n)\frac{(a;q)_{n}(xb;q)_{n}}{(q;q)_{n}}z^n.
    \end{equation}
\end{corollary}
\begin{corollary}
From  Theorem \ref{theo_repre_DHP}, we have the following representation for the Exton polynomials
    \begin{equation}
        {}_{2}\opE_{1}\left(x,y;
        \begin{array}{c}
             z,0 \\
             az
        \end{array};
        q,b
        \right)=\frac{(z;q)_{\infty}}{(xb,az;q)_{\infty}}\sum_{n=0}^{\infty}\EE_{q}(ybq^n)\frac{(a;q)_{n}(xb;q)_{n}}{(q;q)_{n}}z^n,
    \end{equation}
where $\vert bx\vert<1$ and $\vert az\vert<1$.
\end{corollary}

\begin{theorem}[Generalized Heine's transformation formula]\label{theo_heine}
    \begin{equation}
        {}_{1}\rbhs_{1}\left(x,y;u;
        \begin{array}{c}
             z \\
             az
        \end{array};
        q,b
        \right)=\frac{(a,bxz;q)_{\infty}}{(az,bx;q)_{\infty}}{}_{1}\rbhs_{1}\left(a/b,y;u;
        \begin{array}{c}
             z \\
             bxz
        \end{array};
        q,b
        \right),
    \end{equation}
where $0<\vert b\vert<1$.    
\end{theorem}
\begin{proof}
Repeat the proof of the above theorem with
\begin{equation}
    \frac{\ope_{q}(ybq^n,u)}{(aq^n;q)_{\infty}}=\sum_{m=0}^{\infty}\opR_{m}(a/b,y;u|q)\frac{(bq^{n})^m}{(q;q)_{m}}.
\end{equation}
\end{proof}

Some specialization of the Theorem \ref{theo_heine} are
\begin{itemize}
    \item 
    We get Heine's transformation formula Eq.~\eqref{eqn_heine} if we set $u=q$, $x=1$, $y=-c/b$, $z=b$, and $a=b$ in Theorem \ref{theo_heine}.
    \item 
    \begin{equation*}
        {}_{1}\opR_{1}\left(x,y;
        \begin{array}{c}
             z \\
             az
        \end{array};
        q,b
        \right)=\frac{(a,bxz;q)_{\infty}}{(az,bx;q)_{\infty}}{}_{1}\opR_{1}\left(a/b,y;
        \begin{array}{c}
             z \\
             bxz
        \end{array};
        q,b
        \right)
    \end{equation*}
    \item 
    \begin{equation*}
        {}_{1}\opP_{1}\left(x,y;
        \begin{array}{c}
             z \\
             az
        \end{array};
        q,b
        \right)=\frac{(a,bxz;q)_{\infty}}{(az,bx;q)_{\infty}}{}_{1}\opP_{1}\left(a/b,y;
        \begin{array}{c}
             z \\
             bxz
        \end{array};
        q,b
        \right)
    \end{equation*}
    \item 
    \begin{equation*}
        {}_{1}\opS_{1}\left(x,y;
        \begin{array}{c}
             z \\
             az
        \end{array};
        q,b
        \right)=\frac{(a,bxz;q)_{\infty}}{(az,bx;q)_{\infty}}{}_{1}\opS_{1}\left(a/b,y;
        \begin{array}{c}
             z \\
             bxz
        \end{array};
        q,b
        \right)
    \end{equation*}
    \item 
    \begin{equation*}
        {}_{1}\opE_{1}\left(x,y;
        \begin{array}{c}
             z \\
             az
        \end{array};
        q,b
        \right)=\frac{(a,bxz;q)_{\infty}}{(az,bx;q)_{\infty}}{}_{1}\opE_{1}\left(a/b,y;
        \begin{array}{c}
             z \\
             bxz
        \end{array};
        q,b
        \right)
    \end{equation*}
\end{itemize}
\subsection{Mehler's formula for \texorpdfstring{$\opR_{n}(x,y;u|q)$}{r\_n(x,y;u|q)}}
\begin{theorem}[Mehler's formula]\label{theo_mehler}
    \begin{multline}
        \sum_{n=0}^{\infty}\opR_{n}(x,y;u|q)\opR_{n}(z,w;v|q)\frac{t^n}{(q;q)_{n}}\\
        =\frac{1}{(tzx;q)_{\infty}}\sum_{k=0}^{\infty}(uv)^{\binom{k}{2}}\frac{(twy)^k(tzx;q)_{k}}{(q;q)_{k}}\ope_{q}(twxv^k,v)\ope_{q}(tyzu^k,u).
    \end{multline}
\end{theorem}
\begin{proof}
By Eq.~\eqref{eqn_cauchy_gen},
    \begin{align*}
        \sum_{n=0}^{\infty}\opR_{n}(x,y;u|q)\opR_{n}(z,w;v|q)\frac{t^n}{(q;q)_{n}}
        &=\sum_{n=0}^{\infty}\opE(yD_{q}|u)\{x^n\}\opR_{n}(z,w;v|q)\frac{t^n}{(q;q)_{n}}\\
        &=\opE(yD_{q}|u)\left\{\sum_{n=0}^{\infty}\opR_{n}(z,w;v|q)\frac{(xt)^n}{(q;q)_{n}}\right\}\\
        &=\opE(yD_{q}|u)\left\{\frac{\ope_{q}(wxt,v)}{(zxt;q)_{\infty}}\right\}\\
        &=\sum_{n=0}^{\infty}u^{\binom{n}{2}}\frac{y^n}{(q;q)_{n}}D_{q}^n\left\{\frac{\ope_{q}(wxt,v)}{(zxt;q)_{\infty}}\right\}
    \end{align*}
Now, by applying Leibniz's formula Eq.~\eqref{eqn_leibniz}    
    \begin{align*}
        &\sum_{n=0}^{\infty}\opR_{n}(x,y;u|q)\opR_{n}(z,w;v|q)\frac{t^n}{(q;q)_{n}}\\
        &\hspace{1cm}=\frac{1}{(zxt;q)_{\infty}}\sum_{n=0}^{\infty}u^{\binom{n}{2}}\frac{(yt)^n}{(q;q)_{n}}\sum_{k=0}^{n}\qbinom{n}{k}_{q}v^{\binom{k}{2}}w^kz^{n-k}(zxt;q)_{k}\ope_{q}(wtxv^k,v)\\
        &\hspace{1cm}=\frac{1}{(zxt;q)_{\infty}}\sum_{k=0}^{\infty}(uv)^{\binom{k}{2}}\frac{(zxt;q)_{k}(wyt)^k}{(q;q)_{k}}\ope_{q}(wtxv^k,v)\sum_{n=0}^{\infty}u^{\binom{n}{2}}\frac{(u^kytz)^n}{(q;q)_{n}}.
    \end{align*}
This completes the proof.
\end{proof}
\begin{itemize}
    \item 
    If $u=v=1$, $x=z=1$, $y=x$, $w=y$ in the Theorem \ref{theo_mehler}, we obtain the following identity for the Rogers-Szeg\"o polynomials 
\begin{equation*}
    \sum_{n=0}^{\infty}\oph_{n}(x|q)\oph_{n}(y|q)\frac{t^n}{(q;q)_{n}}=\frac{(xyt^2;q)_{\infty}}{(t,x t,yt,x yt;q)_{\infty}},
\end{equation*}
where $\max\{\vert t\vert,\vert xt\vert,\vert yt\vert, \vert xyt\vert\}<1$.
    \item 
    If $u=v=1$, we obtain the following identity \cite{saad} for generalized the Rogers-Szeg\"o polynomials 
\begin{equation*}
    \sum_{n=0}^{\infty}\opr_{n}(x,y)\opr_{n}(z,w)\frac{t^n}{(q;q)_{n}}=\frac{(xyzwt^2;q)_{\infty}}{(txz,txw,tyw,tyz;q)_{\infty}},
\end{equation*}
where $\max\{\vert txz\vert,\vert xwt\vert,\vert tyw\vert, \vert tyw\vert\}<1$
    \item 
    If $u=v=q$, $y\mapsto-y$, and $w\mapsto-w$, we obtain the following identity 
\begin{equation}\label{eqn_mehlerPP}
    \sum_{n=0}^{\infty}\opP_{n}(x,y)\opP_{n}(z,w)\frac{t^n}{(q;q)_{n}}=\frac{(twx,tyz;q)_{\infty}}{(txz;q)_{\infty}}{}_{1}\phi_{2}\left(
        \begin{array}{c}
            txz\\
             txw,tyz
        \end{array};
        q,tyw
        \right).
\end{equation}
    
    \item 
    If $u=q^2$, $v=q^2$, and with the maps $y\mapsto qy$ and $w\mapsto qw$, then
\begin{align}
    &\sum_{n=0}^{\infty}\opS_{n}(x,y;q)\opS_{n}(z,w;q)\frac{t^n}{(q;q)_{n}}\nonumber\\
    &\hspace{2cm}=\frac{1}{(tzx;q)_{\infty}}\sum_{k=0}^{\infty}q^{4k^2}\frac{(tzx;q)_{k}}{(q;q)_{k}}(txy)^k\mathcal{R}_{q}(tyzq^{2k})\mathcal{R}_{q}(txwq^{2k}).
\end{align}

    \item 
    If $u=v=\sqrt{q}$, then
\begin{multline}
    \sum_{n=0}^{\infty}\opE_{n}(x,y;q)\opE_{n}(z,w;q)\frac{t^n}{(q;q)_{n}}\\
    =\frac{1}{(txz;q)_{\infty}}\sum_{k=0}^{\infty}q^{\binom{k}{2}}\frac{(txz;q)_{k}}{(q;q)_{k}}(twy)^k{}_{1}\phi_{1}\left(
        \begin{array}{c}
            0\\
             -\sqrt{q}
        \end{array};
        q,-txw\sqrt{q}^k.
        \right)\\
        \times{}_{1}\phi_{1}\left(
        \begin{array}{c}
            0\\
             -\sqrt{q}
        \end{array};
        q,-tyz\sqrt{q}^k
        \right).
\end{multline}
\end{itemize}

Other identities that arise as special cases of Theorem \ref{theo_mehler} are
\begin{itemize}
    \item If $u=1$, $v=q$, and $w\mapsto-w$, then
\begin{equation}
    \sum_{n=0}^{\infty}\opr_{n}(x,y)\opP_{n}(z,w)\frac{t^n}{(q;q)_{n}}=\frac{(twx;q)_{\infty}}{(tzx,tzy;q)_{\infty}}{}_{1}\phi_{1}\left(
        \begin{array}{c}
            tzx\\
             twx
        \end{array};
        q,twy
        \right).
\end{equation}
    
    \item 
    If $u=1$, $v=q^2$, and $w\mapsto qw$, then
\begin{align}
    &\sum_{n=0}^{\infty}\opr_{n}(x,y)\opS_{n}(z,w;q)\frac{t^n}{(q;q)_{n}}\nonumber\\
    &\hspace{1cm}=\frac{1}{(tyz,tzw;q)_{\infty}}\sum_{k=0}^{\infty}q^{k^2}\frac{(tzx;q)_{k}}{(q;q)_{k}}(twy)^k\mathcal{R}_{q}(twxq^{2k})\\
    &\hspace{1cm}=\frac{1}{(tyz,tzw;q)_{\infty}}\sum_{n=0}^{\infty}q^{n^2}\frac{(twx)^n}{(q;q)_{n}}{}_{1}\phi_{2}\left(
        \begin{array}{c}
            tzx\\
             0,0
        \end{array};
        q,twyq^{2n-1}
        \right).
\end{align}

    \item If $u=1$ and $v=\sqrt{q}$, then
    \begin{multline}
    \sum_{n=0}^{\infty}\opr_{n}(x,y)\opE_{n}(z,w;q)\frac{t^n}{(q;q)_{n}}\\
    =\frac{1}{(tyz,tzx;q)_{\infty}}\sum_{k=0}^{\infty}\sqrt{q}^{\binom{k}{2}}\frac{(tzx;q)_{k}}{(q;q)_{k}}(twy)^k{}_{1}\phi_{1}\left(
        \begin{array}{c}
            0\\
             -\sqrt{q}
        \end{array};
        \sqrt{q},-twxq^{k/2}.
        \right).
\end{multline}
    \item 
    If $u=q^2$, $v=q$, and $y\mapsto qy$ and $w\mapsto-w$, then
    \begin{align}
        &\sum_{n=0}^{\infty}\opS_{n}(x,y;q)\opP_{n}(z,w)\frac{t^n}{(q;q)_{n}}\nonumber\\
        &\hspace{1cm}=\frac{(twx;q)_{\infty}}{(tzx;q)_{\infty}}\sum_{k=0}^{\infty}(-1)^kq^{\frac{3k^2-k}{2}}\frac{(tzx;q)_{k}}{(twx;q)_{k}(q;q)_{k}}(twy)^k\mathcal{R}(tyzq^{2k}).
    \end{align}
    
    \item 
    If $u=q^{2}$, $v=\sqrt{q}$, and with the map $y\mapsto qy$, then
    \begin{align}
        &\sum_{n=0}^{\infty}\opS_{n}(x,y;q)\opE_{n}(z,w;q)\frac{t^n}{(q;q)_{n}}\nonumber\\
        &\hspace{1cm}=\frac{1}{(tzx;q)_{\infty}}\sum_{k=0}^{\infty}q^{\frac{5k^2-k}{2}}\frac{(twy)^k(tzx;q)_{k}}{(q;q)_{k}}\EE_{q}(twxq^{k/2})\RR_{q}(tyzq^{2k}).
    \end{align}

    \item 
    If $u=\sqrt{q}$, $v=q$, and $w\mapsto-w$, then
    \begin{align}
        &\sum_{n=0}^{\infty}\opE_{n}(x,y;q)\opP_{n}(z,w)\frac{t^n}{(q;q)_{n}}\nonumber\\
        &\hspace{1cm}=\frac{(twx;q)_{\infty}}{(tzx;q)_{\infty}}\sum_{k=0}^{\infty}q^{\frac{3}{2}\binom{k}{2}}\frac{(tzx;q)_{k}}{(twx;q)_{k}(q;q)_{k}}(twy)^k\mathcal{E}_{q}(tyzq^{k/2}).
    \end{align}
    
\end{itemize}

\subsection{Srivastava-Agarwal type formulas}

Srivastava and Agarwal \cite{Sri} derived a large number of generating functions of the form
\begin{equation}\label{eqn_SA}
    \sum_{n=0}^{\infty}Q_{n}(x;q)\frac{(\lambda;q)_{n}}{(q;q)_{n}}t^n,
\end{equation}
where $Q_{n}(x;q)$ is a $q$-polynomial. In this section, we will use Mehler's formula, Theorem~\ref{theo_mehler}, to derive the Srivastava-Agarwal generating function of the polynomials $\opR_{n}(x,y;u|q)$. 

\begin{corollary}\label{coro8}
The representation type Srivastava-Agarwal, Eq.~\eqref{eqn_SA}, of $\opR_{n}(x,y;u|q)$ is
\begin{equation}
        \sum_{n=0}^{\infty}\opR_{n}(x,y;u|q)\frac{(a;q)_{n}}{(q;q)_{n}}t^n
        =\frac{(atx;q)_{\infty}}{(tx;q)_{\infty}}\sum_{k=0}^{\infty}(uq)^{\binom{k}{2}}\frac{(-aty)^k(tx;q)_{k}}{(atx;q)_{k}(q;q)_{k}}\ope_{q}(tyu^k,u).
    \end{equation}    
\end{corollary}
\begin{proof}
Set $v=q$, $z=1$, $w=-a$ in Theorem \ref{theo_mehler}. Then
\begin{align*}
    &\sum_{n=0}^{\infty}\opR_{n}(x,y;u|q)\frac{(a;q)_{n}}{(q;q)_{n}}t^n\\
    &\hspace{1cm}=\frac{1}{(tzx;q)_{\infty}}\sum_{k=0}^{\infty}(uq)^{\binom{k}{2}}\frac{(-aty)^k(tzx;q)_{k}}{(q;q)_{k}}(atxq^k;q)_{\infty}\ope_{q}(tyzu^k,u)\\
    &\hspace{1cm}=\frac{(atx;q)_{\infty}}{(tzx;q)_{\infty}}\sum_{k=0}^{\infty}(uq)^{\binom{k}{2}}\frac{(-aty)^k(tzx;q)_{k}}{(atx;q)_{k}(q;q)_{k}}\ope_{q}(tyzu^k,u).
\end{align*}
\end{proof}
\begin{itemize}
    \item
    If $u=1$, $v=q$, $x=z=1$, $y=x$, $w=-y$ in Corollary \ref{coro8}, then we obtain the following result of Srivastava and Agarwal \cite{Sri}
\begin{align*}
    \sum_{n=0}^{\infty}\opr_{n}(x,y)\frac{(a;q)_{n}}{(q;q)_{n}}t^n=\frac{(atx;q)_{\infty}}{(tx,ty;q)_{\infty}}{}_{1}\phi_{1}\left(
        \begin{array}{c}
             tx \\
             atx
        \end{array};
        q,aty
        \right).
\end{align*}
\item 
If $u=q$, $v=q^2$, $w\mapsto-w$, $z=1$, and $y\mapsto qy$ in Corollary \ref{coro8}, then the representation type Srivastava-Agarwal of $\opS_{n}(x,y;q)$ is
\begin{equation}
    \sum_{n=0}^{\infty}\opS_{n}(x,y;q)\frac{(a;q)_{n}}{(q;q)_{n}}t^n=\frac{(atx;q)_{\infty}}{(tx;q)_{\infty}}\sum_{k=0}^{\infty}q^{\frac{3k^2-k}{2}}\frac{(tx;q)_{k}(-aty)^k}{(q;q)_{k}(atx;q)_{k}}\mathcal{R}_{q}(tyq^{2k}).
\end{equation}
\item 
The representation type Srivastava-Agarwal of $\opE_{n}(x,y;q)$ is
\begin{multline}
    \sum_{n=0}^{\infty}\opE_{n}(x,y;q)\frac{(a;q)_{n}}{(q;q)_{n}}t^n\\
    =\frac{(ty;q)_{\infty}}{(tx;q)_{\infty}}\sum_{k=0}^{\infty}(q\sqrt{q})^{\binom{k}{2}}\frac{(tx;q)_{k}}{(q;q)_{k}(ty;q)_{k}}(tay)^k{}_{1}\phi_{1}\left(
        \begin{array}{c}
            0\\
             -\sqrt{q}
        \end{array};
        \sqrt{q},taxq^{k/2}
        \right).
\end{multline}
\end{itemize}

\subsection{Transformation formulas for \texorpdfstring{${}_{2}\Phi_{1}$}{2Phi1}}

Setting $a\mapsto b$ and $b\mapsto a$ in corollaries \ref{coro2} and \ref{coro3}, then matching with Theorem \ref{theo_oper_ax/bx}, we obtain, respectively
\begin{align}
    {}_{2}\Phi_{1}\left(
    \begin{array}{c}
         a/b,0\\
         ax
    \end{array}
    ;q,u,by
    \right)&=\sum_{k=0}^{\infty}\frac{u^{\binom{k}{2}}(by)^k}{(q;q)_{k}(ax;q)_{k}}{}_{1}\Phi_{1}\left(
        \begin{array}{c}
             0 \\
             axq^k
        \end{array};
        q,u,u^kay
        \right)\label{eqn_e-phi2}.\\
        &=\sum_{k=0}^{\infty}\frac{(uq)^{\binom{k}{2}}(-ay)^k}{(q;q)_{k}(ax;q)_{k}}\ope_{q}(u^{k}by,u)\label{eqn_e-phi}.
\end{align}

\begin{corollary}
From Eq.~\eqref{eqn_e-phi2} with $u=1$, we obtain
\begin{equation}
    {}_{2}\phi_{1}\left(
        \begin{array}{c}
             a/b,0 \\
             ax
        \end{array};
        q,ay
        \right)=\sum_{k=0}^{\infty}\frac{(by)^k}{(q;q)_{k}(ax;q)_{k}}{}_{1}\phi_{1}\left(
        \begin{array}{c}
             0 \\
             axq^k
        \end{array};
        q,ay
        \right).
\end{equation}
\end{corollary}

\begin{corollary}[Transformation formula for ${}_{2}\phi_{1}$]
From Eq.~\eqref{eqn_e-phi} with $u=1$, we obtain the important transformation formula from ${}_{2}\phi_{1}$ sum to ${}_{1}\phi_{1}$
\begin{equation}
    {}_{2}\phi_{1}\left(
        \begin{array}{c}
             a/b,0 \\
             ax
        \end{array};
        q,by
        \right)=\frac{1}{(by;q)_{\infty}}{}_{1}\phi_{1}\left(
        \begin{array}{c}
             0 \\
             ax
        \end{array};
        q,ay
        \right).
\end{equation}
\end{corollary}

\begin{corollary}
If we set $u=q$ in Eq.~\eqref{eqn_e-phi2}, we obtain the identity
\begin{align}
    {}_{1}\phi_{1}\left(
    \begin{array}{c}
         a/b\\
         ax
    \end{array}
    ;q,-by
    \right)&=\sum_{k=0}^{\infty}\frac{q^{\binom{k}{2}}(by)^k}{(q;q)_{k}(ax;q)_{k}}{}_{0}\phi_{1}\left(
        \begin{array}{c}
             - \\
             axq^k
        \end{array};
        q,-q^kay
        \right).
\end{align}
\end{corollary}

\begin{corollary}[Transformation formula for ${}_{1}\phi_{1}$]
From Eq.~\eqref{eqn_e-phi} with $u=q$, we obtain the important transformation formula from ${}_{1}\phi_{1}$ sum to ${}_{1}\phi_{2}$
    \begin{equation}
        {}_{1}\phi_{1}\left(
    \begin{array}{c}
         a/b\\
         ax
    \end{array}
    ;q,-by
    \right)=(by;q)_{\infty}{}_{1}\phi_{2}\left(
    \begin{array}{c}
         0\\
         ax,by
    \end{array}
    ;q,ay
    \right).    
    \end{equation}
\end{corollary}

\begin{corollary}
From Eqs. \eqref{eqn_e-phi2} and \eqref{eqn_e-phi}, with $u=q^2$, and $y\mapsto qy$, we obtain the identities
\begin{align*}
    {}_{1}\phi_{2}\left(
    \begin{array}{c}
         a/b\\
         ax,0
    \end{array}
    ;q,qby
    \right)&=\sum_{k=0}^{\infty}\frac{q^{k^2}(by)^k}{(q;q)_{k}(ax;q)_{k}}{}_{0}\phi_{2}\left(
        \begin{array}{c}
             - \\
             axq^k,0
        \end{array};
        q,q^{2k}ay
        \right)\\
        &=\sum_{k=0}^{\infty}\frac{q^{\frac{3k^2-k}{2}}(-ay)^k}{(q;q)_{k}(ax;q)_{k}}\RR_{q}(q^{2k}by).
\end{align*}
\end{corollary}

\begin{corollary}
From Eq.~\eqref{eqn_e-phi2} with $u=\sqrt{q}$ we obtain
    \begin{align}
    &{}_{3}\phi_{3}\left(
    \begin{array}{c}
         \sqrt{a/b},-\sqrt{a/b},0\\
         \sqrt{ax},-\sqrt{ax},-\sqrt{q}
    \end{array}
    ;\sqrt{q},-by
    \right)\nonumber\\
    &\hspace{1cm}=\sum_{k=0}^{\infty}\frac{q^{\frac{1}{2}\binom{k}{2}}(by)^k}{(q;q)_{k}(ax;q)_{k}}{}_{1}\phi_{3}\left(
        \begin{array}{c}
             0 \\
             \sqrt{axq^k},-\sqrt{axq^k},-\sqrt{q}
        \end{array};
        \sqrt{q},q^{k/2}ay
        \right).
\end{align}
\end{corollary}
Finally, set $x=z=1$, $y=x$, and $w=y$ in Eq.~\eqref{eqn_mehlerPP}. Then we have the corollary.
\begin{corollary}[Transformation formula for ${}_{1}\phi_{2}$]
For $\vert t\vert<1$, then
\begin{align}
    {}_{1}\phi_{2}\left(
        \begin{array}{c}
             t \\
             tx,ty
        \end{array};
        q,txy
        \right)=\frac{(t;q)_{\infty}}{(tx,ty;q)_{\infty}}{}_{2}\phi_{1}\left(
        \begin{array}{c}
             x,y \\
             0
        \end{array};
        q,t
        \right).
\end{align}    
\end{corollary}

\subsection{Rogers type formulas}

\begin{theorem}[Rogers formula]\label{theo_rogers}
    \begin{multline}
        \sum_{n=0}^{\infty}\sum_{m=0}^{\infty}\opR_{n+m}(x,y;u|q)\frac{v^{\binom{n}{2}}w^{\binom{m}{2}}t^{n}s^{m}}{(q;q)_{n}(q;q)_{m}}\\
        =\sum_{k=0}^{\infty}\sum_{n=0}^{\infty}(uv)^{\binom{k}{2}}(uw)^{\binom{n}{2}}\frac{(ty)^k}{(q;q)_{k}}\frac{(u^ksy)^{n}}{(q;q)_{n}}\ope_{q}(tv^kx,v)\ope_{q}(sq^{k}w^{n}x,w).
    \end{multline}
\end{theorem}
\begin{proof}
    \begin{align*}
        &\sum_{n=0}^{\infty}\sum_{m=0}^{\infty}\opR_{n+m}(x,y;u|q)\frac{v^{\binom{n}{2}}w^{\binom{m}{2}}t^{n}s^{m}}{(q;q)_{n}(q;q)_{m}}\\
        &\hspace{1cm}=\sum_{n=0}^{\infty}\sum_{m=0}^{\infty}\opE(yD_{q}|u)\{x^{n+m}\}\frac{v^{\binom{n}{2}}w^{\binom{m}{2}}t^{n}s^{m}}{(q;q)_{n}(q;q)_{m}}\\
        &\hspace{1cm}=\opE(yD_{q}|u)\left\{\sum_{n=0}^{\infty}v^{\binom{n}{2}}\frac{(tx)^{n}}{(q;q)_{n}}\sum_{m=0}^{\infty}w^{\binom{m}{2}}\frac{(sx)^{m}}{(q;q)_{m}}\right\}\\
        &\hspace{1cm}=\sum_{k=0}^{\infty}\sum_{n=0}^{\infty}(uv)^{\binom{k}{2}}(uw)^{\binom{n}{2}}\frac{(ty)^k}{(q;q)_{k}}\frac{(u^ksy)^{n}}{(q;q)_{n}}\ope_{q}(tv^kx,v)\ope_{q}(sq^{k}w^{n}x,w).
    \end{align*}
\end{proof}

\begin{corollary}\label{coro9}
If $v=w=1$ in Theorem \ref{theo_rogers}, then
\begin{align}
    &\sum_{n=0}^{\infty}\sum_{m=0}^{\infty}\opR_{n+m}(x,y;u|q)\frac{t^{n}s^{m}}{(q;q)_{n}(q;q)_{m}}\nonumber\\
    &\hspace{2cm}=\frac{1}{(tx,sx;q)_{\infty}}\sum_{k=0}^{\infty}u^{\binom{k}{2}}\frac{(sx;q)_{k}}{(q;q)_{k}}(ty)^k\ope_{q}(u^{k}sy,u).
\end{align}
\end{corollary}

\begin{itemize}
    \item
    The Rogers formula for $\oph_{n}(x|q)$ is:
\begin{equation*}
    \sum_{n=0}^{\infty}\sum_{m=0}^{\infty}\oph_{n+m}(x|q)\frac{t^n}{(q;q)_{n}}\frac{s^m}{(q;q)_{m}}=\frac{(x st;q)_{\infty}}{(t,x t,s,xs;q)_{\infty}},\ \max\{\vert s\vert,\vert t\vert,\vert xs\vert,\vert xt\vert\}<1.
\end{equation*}

    \item 
    If $u=1$, we obtain the following results in \cite{saad}
\begin{align*}
    \sum_{n=0}^{\infty}\sum_{m=0}^{\infty}\mathrm{r}_{n+m}(x,y)\frac{t^{n}s^{m}}{(q;q)_{n}(q;q)_{m}}
        &=\frac{(stxy;q)_{\infty}}{(tx,sx,sy,ty;q)_{\infty}}.
\end{align*}
    
    \item
    If $u=q$, then we get the following result in \cite{saad2},
\begin{align*}
    &\sum_{n=0}^{\infty}\sum_{m=0}^{\infty}\opP_{n+m}(x,y)\frac{t^{n}s^{m}}{(q;q)_{n}(q;q)_{m}}=\frac{(sy;q)_{\infty}}{(tx,sx;q)_{\infty}}{}_{1}\phi_{1}\left(
        \begin{array}{c}
             sx \\
             sy
        \end{array};
        q,ty
        \right)
\end{align*}
where $\vert tx\vert<1$, $\vert sx\vert<1$.

    \item 
    If $u=q^2$ and mapping $y\mapsto qy$, then
\begin{align}
    &\sum_{n=0}^{\infty}\sum_{m=0}^{\infty}\opS_{n+m}(x,y;q)\frac{t^{n}s^{m}}{(q;q)_{n}(q;q)_{m}}\nonumber\\
    &\hspace{2cm}=\frac{1}{(tx,sx;q)_{\infty}}\sum_{k=0}^{\infty}q^{k^2}\frac{(sx;q)_{k}}{(q;q)_{k}}(ty)^k\RR_{q}(q^{2k}sy).
\end{align}

    \item 
    If $u=\sqrt{q}$, then
    \begin{align}
    &\sum_{n=0}^{\infty}\sum_{m=0}^{\infty}\opE_{n+m}(x,y;q)\frac{t^{n}s^{m}}{(q;q)_{n}(q;q)_{m}}\nonumber\\
    &\hspace{2cm}=\frac{1}{(tx,sx;q)_{\infty}}\sum_{k=0}^{\infty}\sqrt{q}^{\binom{k}{2}}\frac{(sx;q)_{k}}{(q;q)_{k}}(ty)^k\EE_{q}(q^{k/2}sy).
\end{align}
\end{itemize}

\begin{corollary}\label{coro10}
If $v=1,w=q$ in Theorem \ref{theo_rogers}, then
\begin{align}
    &\sum_{n=0}^{\infty}\sum_{m=0}^{\infty}(-1)^m\opR_{n+m}(x,y;u|q)\frac{q^{\binom{m}{2}}t^{n}s^{m}}{(q;q)_{n}(q;q)_{m}}\nonumber\\
        &\hspace{2cm}=\frac{(sx;q)_{\infty}}{(tx;q)_{\infty}}\sum_{k=0}^{\infty}u^{\binom{k}{2}}\frac{(ty)^k}{(sx;q)_{k}(q;q)_{k}}{}_{1}\Phi_{1}\left(
        \begin{array}{c}
             0 \\
             sxq^k
        \end{array};
        q,u,u^ksy
        \right).
\end{align}
\end{corollary}

\begin{itemize}
    \item
If $u=1$ in Corollary \ref{coro10}, then
\begin{align}
    &\sum_{n=0}^{\infty}\sum_{m=0}^{\infty}(-1)^m\opr_{n+m}(x,y)\frac{q^{\binom{m}{2}}t^{n}s^{m}}{(q;q)_{n}(q;q)_{m}}\nonumber\\
        &\hspace{2cm}=\frac{(sx;q)_{\infty}}{(tx;q)_{\infty}}\sum_{k=0}^{\infty}\frac{(ty)^k}{(sx;q)_{k}(q;q)_{k}}{}_{1}\phi_{1}\left(
        \begin{array}{c}
             0 \\
             sxq^k
        \end{array};
        q,sy
        \right).
\end{align}

    \item 
If $u=q$ in Corollary \ref{coro10}, then
\begin{align}
    &\sum_{n=0}^{\infty}\sum_{m=0}^{\infty}(-1)^m\opP_{n+m}(x,y)\frac{q^{\binom{m}{2}}t^{n}s^{m}}{(q;q)_{n}(q;q)_{m}}\nonumber\\
        &\hspace{2cm}=\frac{(sx;q)_{\infty}}{(tx;q)_{\infty}}\sum_{k=0}^{\infty}q^{\binom{k}{2}}\frac{(ty)^k}{(sx;q)_{k}(q;q)_{k}}{}_{0}\phi_{1}\left(
        \begin{array}{c}
             - \\
             sxq^k
        \end{array};
        q,-q^ksy
        \right).
\end{align}

    \item 
If $u=q^2$ and mapping $y\mapsto qy$ in Corollary \ref{coro10}, then
\begin{align}
    &\sum_{n=0}^{\infty}\sum_{m=0}^{\infty}(-1)^m\opS_{n+m}(x,y;q)\frac{q^{\binom{m}{2}}t^{n}s^{m}}{(q;q)_{n}(q;q)_{m}}\nonumber\\
        &\hspace{2cm}=\frac{(sx;q)_{\infty}}{(tx;q)_{\infty}}\sum_{k=0}^{\infty}q^{k^2}\frac{(ty)^k}{(sx;q)_{k}(q;q)_{k}}{}_{0}\phi_{2}\left(
        \begin{array}{c}
             - \\
             sxq^k,0
        \end{array};
        q,q^{2k+1}sy
        \right).
\end{align}

    \item 
If $u=\sqrt{q}$ in Corollary \ref{coro10}, then
    \begin{align}
    &\sum_{n=0}^{\infty}\sum_{m=0}^{\infty}(-1)^m\opE_{n+m}(x,y;q)\frac{q^{\binom{m}{2}}t^{n}s^{m}}{(q;q)_{n}(q;q)_{m}}\nonumber\\
        &=\frac{(sx;q)_{\infty}}{(tx;q)_{\infty}}\sum_{k=0}^{\infty}\sqrt{q}^{\binom{k}{2}}\frac{(ty)^k}{(sx;q)_{k}(q;q)_{k}}{}_{1}\phi_{3}\left(
        \begin{array}{c}
             0 \\
             \sqrt{sxq^k},-\sqrt{sxq^k},-\sqrt{q}
        \end{array};
        \sqrt{q},-q^{k/2}sy
        \right).
\end{align}
\end{itemize}

\begin{corollary}\label{coro24}
If $v=w=q$ in Theorem \ref{theo_rogers}, then
\begin{align}
    &\sum_{n=0}^{\infty}\sum_{m=0}^{\infty}(-1)^{n+m}\opR_{n+m}(x,y;u|q)\frac{q^{\binom{n}{2}+\binom{m}{2}}t^{n}s^{m}}{(q;q)_{n}(q;q)_{m}}\nonumber\\
        &\hspace{1cm}=(tx,sx;q)_{\infty}\sum_{k=0}^{\infty}(qu)^{\binom{k}{2}}\frac{(-ty)^k}{(tx;q)_{k}(sx;q)_{k}(q;q)_{k}}{}_{1}\Phi_{1}\left(
        \begin{array}{c}
             0 \\
             sxq^k
        \end{array};
        q,u,u^ksy
        \right).
\end{align}
\end{corollary}

\begin{itemize}
    \item 
    If $u=1$ in Corollary \ref{coro24}, then
\begin{align}
    &\sum_{n=0}^{\infty}\sum_{m=0}^{\infty}(-1)^{n+m}\opr_{n+m}(x,y)\frac{q^{\binom{n}{2}+\binom{m}{2}}t^{n}s^{m}}{(q;q)_{n}(q;q)_{m}}\nonumber\\
        &\hspace{2cm}=(tx,sx;q)_{\infty}\sum_{k=0}^{\infty}q^{\binom{k}{2}}\frac{(-ty)^k}{(tx;q)_{k}(sx;q)_{k}(q;q)_{k}}{}_{1}\phi_{1}\left(
        \begin{array}{c}
             0 \\
             sxq^k
        \end{array};
        q,sy
        \right).
\end{align}

    \item 
If $u=q$ in Corollary \ref{coro24}, then
\begin{align}
    &\sum_{n=0}^{\infty}\sum_{m=0}^{\infty}(-1)^{n+m}\opP_{n+m}(x,y)\frac{q^{\binom{n}{2}+\binom{m}{2}}t^{n}s^{m}}{(q;q)_{n}(q;q)_{m}}\nonumber\\
        &\hspace{2cm}=(tx,sx;q)_{\infty}\sum_{k=0}^{\infty}q^{k^2}\frac{(-ty)^k}{(tx;q)_{k}(sx;q)_{k}(q;q)_{k}}{}_{0}\phi_{1}\left(
        \begin{array}{c}
             - \\
             sxq^k
        \end{array};
        q,-q^ksy
        \right).
\end{align}

    \item 
If $u=q^2$ and mapping $y\mapsto qy$ in Corollary \ref{coro24}, then
\begin{align}
    &\sum_{n=0}^{\infty}\sum_{m=0}^{\infty}(-1)^{n+m}\opS_{n+m}(x,y;q)\frac{q^{\binom{n}{2}+\binom{m}{2}}t^{n}s^{m}}{(q;q)_{n}(q;q)_{m}}\nonumber\\
        &=(tx,sx;q)_{\infty}\sum_{k=0}^{\infty}q^{\frac{3k^2-k}{2}}\frac{(-ty)^k}{(tx;q)_{k}(sx;q)_{k}(q;q)_{k}}{}_{0}\phi_{2}\left(
        \begin{array}{c}
             - \\
             sxq^k,0
        \end{array};
        q,q^{2k+1}sy
        \right).
\end{align}

    \item 
If $u=\sqrt{q}$ in Corollary \ref{coro24}, then
    \begin{align}
    &\sum_{n=0}^{\infty}\sum_{m=0}^{\infty}(-1)^{n+m}\opE_{n+m}(x,y;q)\frac{q^{\binom{n}{2}+\binom{m}{2}}t^{n}s^{m}}{(q;q)_{n}(q;q)_{m}}\nonumber\\
        &\hspace{1cm}=(tx,sx;q)_{\infty}\sum_{k=0}^{\infty}(q\sqrt{q})^{\binom{k}{2}}\frac{(-ty)^k}{(tx;q)_{k}(sx;q)_{k}(q;q)_{k}}\nonumber\\
        &\hspace{4cm}\times{}_{1}\phi_{3}\left(
        \begin{array}{c}
             0 \\
             \sqrt{sxq^k},-\sqrt{sxq^k},-\sqrt{q}
        \end{array};
        \sqrt{q},q^{k/2}sy
        \right).
\end{align}
\end{itemize}

%\section{Statements and Declarations}

%\subsection{Conflict of Interest}
%We have no conflict of interest to disclose.

%\subsection{Data availability}
%The author confirms that the manuscript does not use known data.

%\subsection{Declaration of non-funding}
%The authors received no support from any organization for the submitted work.

%%% REFERENCES %%%


{\small
    \begin{thebibliography}{1}
        \bibitem{gasper}
 G.~Gasper and M.~Rahman. 
        \newblock {\em Basic Hypergeometric Series}.
        \newblock $2^{nd}$ ed. Cambridge University Press, Cambridge, MA, 1990.
        
        \bibitem{Sri}
H.~Srivastava and A.~Agarwal.
        \newblock Generating functions for a class of $q$-polynomials.
        \newblock {\em Annali di Matematica Pura ed Applicata}, 154(4):99--109, 1989.
        
        \bibitem{jag}
 R. Jagannathan and  R. Sridhar.
        \newblock $(p,q)$-Rogers-Szeg\"o polynomials and the $(p,q)$-oscillator.
        \newblock In {\em The Legacy of Alladi Ramakrishnan in the Mathematical
          Sciences}, pages 491--501. Springer New York, 2010.
        
        \bibitem{ramanujan}
        S.~Ramanujan.
        \newblock Proof of certain identities in combinatory analysis.
        \newblock {\em Mathematical Proceedings of the Cambridge Philosophical Society}, 19(1):214--216, 1919.
        
        \bibitem{rogers}
        L.~J. Rogers.
        \newblock Second memoir on the expansion of certain infinite products.
        \newblock {\em Proceedings of the London Mathematical Society}, 25(1):318--343, 1894.
        
        \bibitem{saad}
        H.~L. Saad and M.~Abdul.
        \newblock The $q$-exponential operator and generalized Rogers-Szeg\"o
          polynomials.
        \newblock {\em Journal of Advances in Mathematics}, 8(1):1440--1455, 2014.
        
        \bibitem{saad2}
        H.~L. Saad and A.~A. Sukhi.
        \newblock The $q$-exponential operator.
        \newblock {\em Applied Mathematical Sciences}, 7(1):6369--6380, 2013.
        
        \bibitem{Al}
W. Al-Salam and M. Ismail.
        \newblock $q$-beta integrals and the $q$-hermite polynomials.
        \newblock {\em Pacific Journal of Mathematics}, 135(2):209--221, 1988.
        
        \bibitem{chen2}
W. Chen and Z-G. Liu. 
        \newblock Parameter augmenting for basic hypergeometric series, II.
        \newblock {\em Journal of Combinatorial Theory, Series A}, 80(1):175--195, 1997.
    \end{thebibliography}
}
\end{document}